\documentclass[11pt]{article}

\usepackage{epsfig}

\usepackage{enumerate}
\usepackage{amsfonts}
\usepackage{amssymb}
\usepackage{amsmath}   
\usepackage{amsthm}
\usepackage{amscd}
\usepackage{amstext}
\usepackage{latexsym}

\usepackage{xy}
\usepackage{footmisc}

\input xy
\xyoption{all}

\newcommand{\ZZ}{\mathbf{Z}} 

\newcommand{\CC}{\mathbf{C}}

\newcommand{\QQ}{\mathbf{Q}}
\newcommand{\RR}{\mathbf{R}}
\newcommand{\RRR}{ ( \RR \cup \{ -\infty \} ) } 
\newcommand{\PP}{\mathbf{P}}
\newcommand{\JJ}{\mathcal{J}}
\newcommand{\OO}{\mathcal{O}}
\newcommand{\XX}{\mathcal{X}}
\newcommand{\HH}{\mathcal H} 
\newcommand{\LL}{\mathcal L}

\newcommand{\de}{\partial}
\newcommand{\iddb}{\sqrt{-1} \partial \overline{\partial}}
\newcommand{\ii}{\sqrt{-1}}
\newcommand{\db}{\overline{\partial}}
\newcommand{\al}{\alpha}
\newcommand{\Gm}{\Gamma}

\newcommand{\om}{\Omega}
\newcommand{\og}{\omega}
\newcommand{\mg}{\overline \omega}
\newcommand{\zt}{\zeta}
\newcommand{\ztb}{\overline{\zeta}}
\newcommand{\ep}{\epsilon}
\newcommand{\zsum}{ \sum_{i=1}^k |z_i^{(\el)}|^2 }
\newcommand{\zsuml}{ \sum_{i=1}^k |z_i^{(\el')}|^2 }

\newcommand{\zsumo}{ \sum_{i=1}^k |z_i|^2 }
\newcommand{\zsumoe}{ \zsumo + \ep^2 }
\newcommand{\el}{\ell}

\newcommand{\ld}{\lambda}
\newcommand{\kap}{\kappa}
\newcommand{\varep}{\varepsilon}
\newcommand{\qa}{\quad}
\newcommand{\noi}{\noindent}

\newcommand{\st}{\widetilde{s}}

\newcommand{\li}{\langle\!\langle}
\newcommand{\ri}{\rangle\!\rangle}
\newcommand{\zb}{\bar{z}}
\newcommand{\ub}{\bar{u}}
\newcommand{\ve}{\vert  } 

\newcommand{\lan}{\langle}
\newcommand{\ran}{\rangle}
\newcommand{\rrow}{\rightarrow}

\newcommand{\xr}{ X_{\reg}}
\newcommand{\xs}{ X_{\sing}}
\newcommand{\zr}{ Z_{\reg}}
\newcommand{\zs}{ Z_{\sing}}

\providecommand{\abs}[1]{\lvert#1\rvert}
\providecommand{\norm}[1]{\lVert#1\rVert}
\providecommand{\inner}[1]{\li#1\ri}
\providecommand{\inp}[1]{\langle#1\rangle}
\providecommand{\tit}[1]{\noindent {{\textbf{#1}}}}

\theoremstyle{plain}

\newtheorem{theorem}{Theorem}[section]

\newtheorem{lemma}[theorem]{Lemma}
\newtheorem{corollary}[theorem]{Corollary}
\newtheorem{proposition}[theorem]{Proposition}
\newtheorem{definition}[theorem]{Definition}

\newtheorem{question}[theorem]{Question}

\theoremstyle{remark}

\newtheorem{item1}[theorem]{Remark}

\newtheorem{remark}[theorem]{Remark}
\newtheorem{example}[theorem]{\textnormal{\textbf{Example}}}

\DeclareMathOperator{\red}{red}

\DeclareMathOperator{\divisor}{div}

\DeclareMathOperator{\reg}{reg}
\DeclareMathOperator{\Dom}{Dom}

\DeclareMathOperator{\sing}{sing}

\long\def\symbolfootnote[#1]#2{\begingroup\def\thefootnote{\fnsymbol{footnote}}
\footnote[#1]{#2}\endgroup}
\makeatletter
\renewcommand\section{\@startsection {section}{1}{\z@}%
                                   {-3.5ex \@plus -1ex \@minus -.2ex}%
                                   {2.3ex \@plus.2ex}%
                                   {\centering\normalfont\Large\bfseries}}

\makeatother

\begin{document}

\title{\Large  {\textsc{$L^2$ extension of adjoint line bundle sections}}}

\author{\large \textsc{Dano~Kim}}

\date{}

\maketitle

\small

\begin{abstract}
\noindent We prove an $L^2$ extension theorem of Ohsawa-Takegoshi type for extending holomorphic sections of line bundles from a subvariety which is given as a maximal log-canonical center of a pair and is of general codimension in a projective variety. Our method of proof indicates that such a setting is the most natural one in a sense, for general $L^2$ extension of line bundle sections.

\end{abstract}

\tableofcontents

\normalsize

\section{Introduction} 

 The purpose of this paper is to prove an $L^2$ extension theorem (Theorem~\ref{main}) of Ohsawa-Takegoshi type to lift line bundle sections from a closed subvariety of general codimension of a projective variety. For the moment, let $Z \subset X$ be a complex submanifold of a complex manifold. Let $L$ be a line bundle on $X$ together with a \textit{norm}  $\norm{\cdot}_1$ (see (\ref{wnorm})) for holomorphic sections in $\Gm (X,L)$ and a norm $\norm{\cdot}_2$ for holomorphic sections in $\Gm(Z, L|_Z)$. $L^2$ extension is a statement of the following type (under suitable conditions on the quintuple $\Lambda = (X,Z,L,\norm{\cdot}_1, \norm{\cdot}_2)$ of the above data):
\\

(*) \textit{If a section $s \in \Gm(Z,L|_Z)$ has the finite norm $\norm{s}_2 < \infty$, then there exists a section $\st \in \Gm(X,L)$ such that $\st|_Z = s$ and its norm is bounded by $ \norm{\st}_1 \le C \norm{s}_2 $},
\\
 
\noindent where $C>0$ is a constant independent of $s$ and independent of $L$ (having $L$ within a class of line bundles on $X$ to be specified).

First established by \cite{OT} in a prototypical case, results of $L^2$ extension under various conditions on $\Lambda$ (concerning, for example, 1) $X$  non-compact or compact, 2) $Z$ a hypersurface or of higher codimension, 3) positivity conditions on $L$) were given by \cite{Ma}, \cite{Oh95}, \cite{Siu96}, \cite{B}, \cite{Siu98}, \cite{Siu02}, \cite{D97}, \cite{M}, \cite{MV}, \cite{Var} and others. These results lead to numerous applications in algebraic geometry and complex analysis.
 
 $L^2$ extension theorems are comparable to the vanishing theorem of Kodaira type due to Kawamata, Viehweg and Nadel which has played a fundamental role in complex algebraic geometry. Both of them are consequences of $L^2$ methods (\cite[5.11]{D94}) and are used to extend line bundle sections from a subvariety. 
An algebraic geometer may view $L^2$ extension as using the methods of proof for vanishing to obtain consequences of vanishing, not via sheaf cohomology.

 While the Kodaira-type vanishing theorem requires certain strict positivity condition of the involved line bundle, the possibility of $L^2$ extension to work under weaker positivity condition than vanishing and therefore to give stronger results was first realized by Siu~\cite{Siu02} (see (\ref{OTS})). From the viewpoint of $L^2$ methods, this is not too surprising since even the first instance of $L^2$ extension was only possible with the innovation due to \cite{OT} of twisting $\db$ operators, while vanishing follows from the earlier version of $L^2$ methods for usual $\db$ operators.

 We want to see what this exciting new possibility from \cite{Siu02} will lead to in general beyond the particular setting of a local family in (\ref{OTS}). On one hand, we simply ask what would be the most general condition on the quintuple $\Lambda$ for (*) above to hold. On the other hand, from the extensive experience of applying the vanishing theorem in algebraic geometry, we expect that the setting of a \textit{log-canonical center} (Section 3.1) may be relevant. We will see how these two viewpoints fit together. Let us make the former question precise. It is natural to replace the line bundle $L$ by an adjoint line bundle $K_X + L$ and take $\norm{\cdot}_1$ as an \textit{adjoint norm} (\ref{wnorm}).

\begin{question}\label{mainq}

 Let $ Z \subset X $ be a (smooth) irreducible subvariety of a (smooth) complex projective variety. For which quintuple $\Lambda = (X,Z,K_X+L,\norm{\cdot}_1, \norm{\cdot}_2)$, does there exist a constant $C_{\Lambda} > 0$ such that the following holds ?: 
  
\noindent If $(B,b)$ is any singular hermitian line bundle on $X$ with nonnegative curvature current and $s \in \Gm(Z,(K_X+L)|_Z + B|_Z)$ is any holomorphic section satisfying 
 
\begin{align*} 
 \norm{s}_{2,b} < \infty  , 
\end{align*}
	
\noindent \text{then there exists a holomorphic section} $ \st \in \Gm(X,(K_X + L) + B)$ \text{such that} $ \st|_Z = s $ \text{and}  

\begin{align*}
   \norm{\st}_{1,b}  \le C_{\Lambda} \norm{s}_{2,b}  
\end{align*}

\noindent where $\norm{\cdot}_{1,b}$ and $\norm{\cdot}_{2,b}$ are the norms given by multiplication of the original metrics with $b$. The constant $ C_{\Lambda} $ is independent of $(B,b)$ and of the section $s$.  
 
\end{question}

\noindent Recall that a line bundle $B$ has a hermitian metric $b$ with nonnegative curvature current if and only if $B$ is pseudo-effective (\cite{D94}). So the statement in Question~\ref{mainq} says that if $L$ is a \textit{right} line bundle, then adding any pseudo-effective $(B,b)$, $L+B$ also admits the $L^2$ extension. Though Question~\ref{mainq} is for arbitrary $(X,Z,K_X+L)$, the setting of an lc center enters the picture through the following two main obstacles to be addressed for the question. 
\\

 The first obstacle is that we need to identify the optimal positivity condition to put on $L$ with respect to (the normal bundle of) $Z$. Applying Twisted Basic Estimate (\cite{MV}, \cite{Siu02}) to $Z$ of general codimension, the positivity of $L$ we need turns out to be the existence of a real-valued function $\lambda$ on (a Zariski open subset of) $X$ satisfying the positivity conditions \eqref{condition1}, \eqref{condition2}.  For a general subvariety $Z$, we do not see a natural way to guarantee the existence of such a function. But when $Z$ is a maximal lc center (of $D \sim  L$), the function $\lambda$ is given by using global multi-valued holomorphic sections of $L$ generating the multiplier ideal sheaf $\JJ(D)$ (on a Zariski open subset of X) by Siu's global generation theorem of multiplier ideal sheaves \eqref{siu global}.   
 
 That is, the positivity of $L$ we need against $Z$ is essentially the existence of a $\QQ$-divisor $D$ such that $(X,D)$ has $Z$ as an lc center. This is in accordance with the heuristic that when we try to find such a $\QQ$-divisor $D$ linearly equivalent to given $L$, we need $D$ to have high multiplicity along $Z$, which will become more difficult when the normal bundle of $Z$ is higher.

 The second obstacle for Question~\ref{mainq} is that it is most natural to have the norm $\norm{\cdot}_2$ also as an adjoint norm, which means that we need a particular choice of a singular hermitian metric $h$ of the line bundle $M:= -{K_Z} + (K_X+L)|_Z$. For a general subvariety $Z$, $M$ does not seem to be a remarkable line bundle coming with such a particular metric. But when $Z$ is an lc center, the fundamental subadjunction result of \cite{Ka98} gives an effective $\QQ$-divisor $h_Z \sim M$ with certain properties, under some additional conditions. (Note that such effectiveness of the line bundle $M$ is already highly non-trivial.)  Essentially, the metric associated to $h_Z$ turns out to give the metric we need in the proof of our $L^2$ extension since it gives the first main inequality I $\ge$ I* via (\ref{kltklt}) (see also (\ref{rem2}) (a)).  
 
 To sort out the idea involved here, first consider the following simple approach of using $Z$ to obtain a non-zero holomorphic section of $K_X + L$ on $X$. On one hand, (a) we need a section $\sigma$ of $(K_X+L)|_Z$ from some inductive hypothesis on $Z$ and on the other hand, (b) we need to extend $\sigma$ to $X$. While the subadjunction \cite{Ka98} with $h_Z$ itself is concerned with the former step (a), only a candidate divisor (not necessarily effective) for $h_Z$ is enough to define our metric $h$ of $M$ for the purpose of the latter extension (b). We call this particular metric $(M,h)$ (which is given by $Q(R_1)$ in the setting of  a \textit{refined log-resolution} \eqref{QR}) a \textbf{Kawamata metric} (\ref{kawamata}) of the log-canonical center. We only need $h$ to be defined up to a proper closed subset of $Z$ and therefore it is enough to have it defined on the level of $Z'$, birational over $Z$. The definition of $h$ does not use the positivity result  \cite[Theorem 2]{Ka98} which was the main technical part dealing with the issue of $h_Z \ge 0$ on the level of $Z$.     
 
\noi These two obstacles and their resolution give our main theorem, which is an answer to Question~\ref{mainq}. 
\\

\noindent \textbf{Main Theorem} (see \textbf{Theorem~\ref{main}} for the full statement)  \textit{Let $X$ be a normal projective variety and $Z \subset X$ a subvariety which is not contained in $\xs$ and is a maximal log-canonical center of a log-canonical pair $(X,D)$. Let $K_X + L$ be the $\QQ$-line bundle $\OO(K_X + D + A)$ for any ample $\QQ$-line bundle $A$ and let $ \norm{\cdot}_1 $ be the adjoint norm given by a Kawamata metric on $Z$ (\ref{kawamata}). Then there exist an adjoint norm $\norm{\cdot}_2$ and a constant $C_{\Lambda}$ such that we have the $L^2$ extension as in Question~\ref{mainq} for those $B$ with $K_X + L + B$ being an integral line bundle. }
\qa
\\
 
\noi Note that, even though we formulated Question~\ref{mainq} for the quintuple $\Lambda$, it turns out that for the triple $(X,Z, K_X + L)$ coming from an lc center, there are natural choices of $\norm{\cdot}_1$ and $\norm{\cdot}_2$.  
 
\noi We give here a short outline of the proof (Section~4.2). Following \cite{Siu98}, \cite{Siu02} and using (\ref{montel}), (\ref{riemann extend}), we reduce obtaining the extended holomorphic section on $X$ to solving a $\db$ equation \eqref{mainequation} on each member of an increasing sequence of bounded Stein open subsets of $X \setminus H$ where $H$ is a hyperplane section we choose. Solving the $\db$ equation is equivalent to showing the inequality \eqref{apriori}. Using Cauchy-Schwarz, inequality \eqref{apriori} reduces to two inequalities I $\ge$ I* and II $\ge$ II*. Up to this point, the setup is valid for a general subvariety $Z \subset X$. We proceed to prove I $\ge$ I* and II $\ge$ II* using the condition that $Z$ is a maximal lc center. We use the main property (\ref{kltklt}) of the Kawamata metric for I $\ge$ I* and use the $\lambda$ function satisfying \eqref{condition1}, \eqref{condition2} for II $\ge$ II*.  See also (\ref{remark46}).  Note that \eqref{def lambda} with II $\ge$ II* is already a strong indication that the setting of an lc center is relevant, but it only works when combined with \mbox{I $\ge$ I*} and (\ref{kltklt}) which is another fundamental property of an lc center and is based on the work of Section 3.

 To put the statement and the proof of (\ref{main}) in the right perspective, we have the following remarks. 

\begin{item1}\label{rem2}  \textit{More general statements.}

\begin{enumerate}[(a)]

\item

 The proof of Theorem~\ref{main} works for more general $(X, Z, K_X+L)$ as far as the following two are satisfied (in the setting of Section~4.2): 
 
1. There exist $\ld_t = \ld(t, \nu, \ep) : \om_t \to \RR_{\ge 1} $ satisfying \eqref{condition1} and \eqref{condition2} and having $-\ld_t$ uniformly bounded above. 
 
2. There exists a metric $h$ of $M$ such that \eqref{given} for $s$ implies \eqref{Istar} for $\st_\el$.

\item 
   
 In Theorem~\ref{main}, the use of an arbitrarily small ample $\QQ$-line bundle  $A$ is completely limited to construction of the $\lambda$ function \eqref{def lambda} for which we use the global generation of \eqref{siu global}. This is different from situations where such $A$ is used to apply, for example, a Kodaira-type vanishing theorem.

\item

 Since we only need $\lambda$ on $X \setminus H$, we do not need the pair $(X,D)$ to be log-canonical along every irreducible component of the non-klt locus of $(X,D)$. We only need $Z$ to be the maximal lc center in the sense of \cite[Sec.2.3]{T06}. See also Section 5.

\end{enumerate}

\end{item1}

\begin{item1} \textit{Comparison with the case of a hypersurface $Z$.}

 In various applications, the setting of a log-canonical center (Section~\ref{center}) is used to formalize the following idea: we study an adjoint line bundle $K_X+L$ by constructing an effective $\QQ$-divisor $D \sim_{\QQ} L$ having non-integrable singularity along a subvariety $Z$ (a log-canonical center) and use the inductive approach of restricting the line bundle $K_X+L$ to $Z$. The dimension of $Z$ can be basically any number between $0$ and $\dim X - 1$. To use the inductive approach, we need to be able to extend line bundle sections from $Z$ to $X$. Our $L^2$ extension for $Z$ of general codimension (\ref{main}) does this by (a) formulating and solving the $\db$ equation \eqref{mainequation} (following the line of \cite{Ma}, \cite{D97}, \cite{Oh95}). As we will see, it is a very natural approach with which one can use the condition on $(X, Z, K_X + L)$ most directly.  On the other hand, this is clearly different from either

\noi (b) having $Z$ as a complete intersection of hypersurfaces and successively applying extension from a hypersurface, or
 
\noi (c) using a log-resolution $f: X' \to X$ of $(X,D)$ and applying extension to $X'$ from a particular exceptional hypersurface lying over $Z$. 
 
 It is clear that the approach (b) does not go far for projective varieties. The approach (c) has been profitably used as far as it worked, but it seems to have limitation in that the statement and proof of the hypersurface extension results (either with the Kodaira-type vanishing or with $L^2$ extension \cite{Var}, \cite{BP}) are independent of the geometric setting involving an lc center. Attempts to simply weaken positivity conditions in hypersurface extension outside such a geometric setting lead to simple counterexamples (Example~\ref{ex1}). 
 One should somehow be able to work with only those hypersurfaces appearing as exceptional divisors over lc centers, or much more naturally, should try to extend sections directly from $Z$ as in (a). We note that with (a), we have a natural ready-made choice of metrics for $L^2$ extension, which is not the case with (c).

 From the viewpoint of $L^2$ methods, we need to solve a $\db$ equation at some point, with any approach for extension. Unlike all the previous cases, we use the condition of an lc center before solving a $\db$ equation, not after.

\end{item1}

\begin{item1} \textit{For more algebraically inclined readers.}

 The statement of extension with the condition ($\dagger$) $ \norm{\st}_1 \le C \norm{s}_2 $ as in (*) is a surprisingly natural and effective one for algebraic geometry, though it might not look so at first glance. Applying (\ref{montel}) and (\ref{riemann extend}), extension on each Stein $\om_t$ with ($\dagger$) ($C$ being independent of $t$) gives a global extension on projective $X$, while extension on a Stein manifold without ($\dagger$) is trivial. It would be best to view this approach originally due to \cite{Siu98} as a fundamentally different way from the method of sheaf cohomology, in obtaining a global holomorphic section from local data.

\end{item1}

\begin{item1} \textit{For more analytically inclined readers.}

 The use of a log-resolution in Section 3 is only to achieve (a)2 of (\ref{rem2}). It is natural to use it even if one starts from an analytic setting, as follows. For simplicity, suppose that $Z$ and $X$ are smooth and that $Z$ is precisely the locus of non-integrable singularity (or, the zero set of the multiplier ideal) of a plurisubharmonic weight $e^{-\varphi}$ of $L$ which is not necessarily given by an effective $\QQ$-divisor. As soon as we use the approximation of $e^{-\varphi}$ by an algebaic divisor \cite{D97} locally or globally, we can use a log-resolution, which in principle will give the information one needs in the original setting.

\end{item1}

\qa
\\

{ \noindent \small \textit{Acknowledgement.}  This is a version of a Ph.D. thesis at Princeton University in 2007. I am very grateful to my advisor Professor J\'anos Koll\'ar for his generous encouragement and support. I would like to thank Professor Osamu Fujino for helpful discussions and Professor Dror Varolin for helpful discussions and reading an earlier draft of this paper. I would also like to thank Professors Yum-Tong Siu, Jean-Pierre Demailly, Christopher Hacon, Lawrence Ein and Bo Berndtsson for answering my questions and Mihai P\u aun and Stephane Druel for pointing out an incorrect statement in an earlier version of this paper (see (\ref{remark1})).
\\

\tit{Notation and Conventions}

\begin{enumerate}

\small

\item{}

 Let $X$ be a reduced complex analytic space. We let $X_{\sing}$ denote the closed subset of singular points in $X$ and let $X_{\reg} := X \setminus X_{\sing}$. When $X$ is an algebraic variety (reduced and irreducible) defined over $\CC$,  we often identify $X$ with its associated complex analytic space and $\xr$ with its associated complex manifold.   
 
\item{}

 Let $X$ be a projective variety and $F$ a $\QQ$-line bundle on $X$ such that $F|_{\xr} \cong K_{\xr} + L$ for a $\QQ$-line bundle $L$ on $\xr$. As a slight abuse of notation, we often denote $F$ on $X$ by $K_X + L$ (\ref{adjlb}).

\item{}
 Let $L$ be a $\QQ$-line bundle (Definition~\ref{Qlb}) on a reduced complex analytic space. With a slight abuse of notation, we use $\Gm (X, L)$ to denote the $\CC$-vector space of multi-valued holomorphic sections of $L$. 

\item{}
 We define and use hermitian metrics (Definition~\ref{metric}) of $\QQ$-line bundles only on a complex manifold, for example, on an open subset of $\xr$. Similarly, we use the multiplier ideal sheaf $\JJ(D)$ of a $\QQ$ divisor $D \ge 0$ and a plurisubharmonic function only on a complex manifold.

\item{}
 
 We use additive notation for tensor products and powers of line bundles and multiplicative notation for hermitian metrics of line bundles. For example, $(L,g)$, $(M,h)$ and $(L+M, g \cdot h)$. 

\item{} 
 
 \textit{lc, snc, psh} are abbreviations for \textit{log-canonical, simple normal crossings, plurisubharmonic}, respectively. A $\QQ$-divisor $D = \sum d_i D_i $ on a complex manifold is said to be \textit{snc} if each $D_i$ is smooth and they intersect everywhere transversally (\cite{Ko97}).

\end{enumerate}

\normalsize

\section{Preliminaries}

\subsection{Singular hermitian metrics}\label{s-metric}

\subsubsection{The first kind}

\normalsize

 Let $X$ be a reduced complex analytic space. An invertible sheaf $L$ on $X$ is identified with a line bundle $L$ on $X$. Sections of the structure sheaf $\OO_X$ are called holomorphic functions on $X$ \cite[p.9]{GR2}. A line bundle $L$ is further identified with (an equivalence class in $H^1(X,\OO_X^*)$ of) a collection of transition functions on an open covering of $X$ \cite[(III, Ex. 4.5)]{H}. Now we define a $\QQ$-line bundle on $X$ (following \cite{AS} and others): 

\begin{definition}\label{Qlb}

 Let $X$ be a reduced complex analytic space. A $\QQ$-line bundle $L$ on $X$ is (an equivalence class of) a collection of holomorphic transition functions $\{ g_{ij} : U_i \cap U_j \to \CC  \}$ on an open covering $\{ U_i \}$  of $X$ such that there exists an integer $m \ge 1$ and $\{ g_{ij}^m \}$ defines a line bundle on $X$ (which we denote by $mL$). 

\end{definition}

\noindent If we can take $m=1$, the $\QQ$-line bundle $L$ is just a line bundle in the usual sense, which we call an integral line bundle. Along with a $\QQ$-line bundle, it is natural to define a multi-valued holomorphic section (following \cite{AS} and others):

\begin{definition}\label{mv} 

 Let $L$ be a $\QQ$-line bundle with transition functions as in Definition~\ref{Qlb} such that $mL$ is an integral line bundle for an integer $m \ge 1$. A \textbf{multi-valued holomorphic section} (or a multi-valued section) $s$ of $L$ is a collection of holomorphic functions $\{ f_i \in \OO_X (U_i) \}$ such that

 $$ g_{ij}^m f_j^m = f_i^m, \quad \forall i,j .$$
 
\end{definition}

\noindent Note that the collection $\{ f_i^m \}$ defines a holomorphic section of the integral line bundle $mL$, in the usual sense. We also note that even when $L$ is an integral line bundle, multi-valued sections we obtain from the definition are more general than the usual holomorphic sections. 
  
 Now we introduce a singular hermitian metric of a $\QQ$-line bundle. In this paper, we define and use a singular hermitian metric only over an open set of $\xr$, in other words, over a complex manifold whereas we use a $\QQ$-line bundle over a reduced complex analytic space. First, we begin with the following general notion of a \textit{hermitian metric}:

\begin{definition}\label{metric}
 
 Let $L$ be a $\QQ$-line bundle on a reduced complex analytic space $X$ as in Definition~\ref{Qlb}. Let $X_0$ be an open subset of smooth points $\xr$.   A \textbf{hermitian metric} of $L$ on $X_0$ is a collection of measurable functions $\{ \al_i: U'_i := U_i \cap X_0 \to \RR \cup \{ \pm \infty \} \}$ such that $ e^{-\al_i} = \abs{g_{ij}}^2 e^{-\al_j} $ on $ U'_i \cap U'_j $. 

\end{definition}

\noindent A smooth hermitian metric of $L$ on $X_0$ is such a collection with each $e^{-\al_i}$ being a positive $C^\infty$ function.   
 Equivalently to the above definition, a hermitian metric $h$ of $L$ is given by $ h = h_0 \cdot e^{-\phi} $ (following S. Takayama in part) where $h_0$ is a smooth hermitian metric of $L$ and $\phi : X_0 \rrow \RR \cup \{ \pm \infty \} $ is any measurable function. Note that $h_0$ can be taken as the $m$-th root of any usual smooth hermitian metric of $mL$ in case $mL$ is an integral line bundle. We call the pair $(L,h)$ a singular hermitian $\QQ$-line bundle (or simply a singular metric, meaning the obvious pair $(L,h)$). The open subset $X_0 \subset \xr$ is called the \textbf{domain} of $(L,h)$. Usually, a singular hermitian metric is defined as a hermitian metric with the condition that the function $\al_i$ is locally integrable for each $i$. Instead of this, we will have two different definitions, a \textit{singular hermitian metric of the first kind} in (\ref{metric1}) and a \textit{singular hermitian metric of the second kind} after (\ref{wnorm}). 
\\

\noi Now when $\al_i \in L^1_{loc} (U_i)$ for each $i$, we define the \textbf{curvature current} $\ii \Theta_h (L)$ of $(L,h)$ to be $\iddb \al_i$ on each $U_i$, which is then a globally well-defined $(1,1)$ current on $X$ (see, for instance, \cite{D94} or \cite{L}, (9.4.19)). Up to upper semicontinuous regularization (\ref{usc reg}), the curvature current is nonnegative if and only if $\al_i$ is \textit{plurisubharmonic} (\ref{psh}) (\textit{psh} for short).

\begin{definition}\label{metric1}      
  
 A hermitian metric $(L,h)$ is called a \textbf{singular hermitian metric of the first kind} if each local weight function $\al_i$ is plurisubharmonic. 

\end{definition}

\noindent Unless otherwise specified, the domain of a singular hermitian metric of the first kind is always assumed to be the largest possible one, that is, $\xr$.

\subsubsection{The second kind and the adjoint norm}\label{anorm}

 Let $X$ be a complex manifold, $(L,h)$ an integral line bundle with a singular hermitian metric of the first kind on $X$ and $s$ a holomorphic section of $K_X+L$. In \cite{Siu98}, Siu defined and used the integral of the absolute-value square of $s$ viewed as a $L$-valued holomorphic $n$ form, denoting the integral by $ \int_X |s|^2 \cdot h $.  We will call it the \textit{adjoint norm} of $s$ with respect to $h$. In this section, we generalize the definition of the adjoint norm using the notion of a singular hermitian metric of the second kind, in order to formulate the $L^2$ extension in the more general setting as in Theorem~\ref{main}.
\\

\noindent Let $X$ be a reduced complex analytic space. The canonical line bundle $K_{\xr}$ on $\xr$ may not extend as a line bundle on the whole of $X$.

\begin{definition}\label{adjlb}

 A $\QQ$-line bundle $F$ on $X$ is said to be an \textbf{adjoint line bundle} if there exists a  $\QQ$-line bundle $L$ on $\xr$ such that $F|_{\xr} \cong K_{\xr} + L $ on $\xr$. 
 
\end{definition}

\noi In general (when $X$ is not normal), there may possibly be more than one $\QQ$-line bundle $L$ one can take. We fix one of them and call it $L$. We will use the slight abuse of notation $K_X + L$ for an adjoint line bundle, where $L$ is understood as a line bundle on $\xr$ as in (\ref{adjlb}).

 Let $(L,h)$ be a hermitian metric with its domain $X_0 \subset \xr$. Since each local weight function $\al_i$ is measurable, the function $e^{-\al_i}$ is also measurable. Let $s$ be a multi-valued holomorphic section of $F$. When restricted to the open set $\xr$, $s$ gives a holomorphic $L$-valued $n$-form on $X_{\reg}$ (where $ n = \dim X$). We will define the adjoint norm of $s$ with respect to $h$ in this setting. 
 
  Let $\xi \in \Gamma(U,L)$ be a local generating section on any given open neighborhood $ U \subset X_0$. Following \cite{Var07}, choose local analytic coordinates $z_1, \cdots, z_n$ in $U$ such that 
 
 $$ s = f(z) \xi \otimes dz_1 \wedge \cdots \wedge dz_n $$ 
 
\noindent where $f$ is a holomorphic function on $U$. Let $\phi$ be a function on $U$ with $ e^{-\phi} = h(\xi, \xi) $, the square length function of $\xi$ with respect to the hermitian metric $h$. The collection of $2n$-forms $ \abs{f(z)}^2 e^{-\phi} (\frac{\sqrt{-1}}{2})^n  dz_1 \wedge d\zb_1 \wedge \cdots \wedge dz_n \wedge d\zb_n $ on each $U$ defines a globally well-defined $2n$-form $\omega$ on $X_0$ (\cite[(5.1.3)]{Var07}).

\begin{definition}\label{wnorm}

 Let $K_X + L$ be an adjoint line bundle~(\ref{adjlb}) on a reduced complex analytic space $X$ and $(L,h)$ a hermitian metric with its domain $X_0 \subset \xr$ as above. Let $ s \in \Gm (X, K_X + L)$ be a multi-valued holomorphic section. The integral $\int_{X_0} \omega$ of $\omega$ given in the above paragraph is called the \textbf{adjoint norm} of $s$ with respect to $h$ and denoted by $ \int_X |s|^2 \cdot h  \;  (= \int_{X_0} |s|^2 \cdot h) $.
     
\end{definition}
 
\noindent Note that $ \int_X |s|^2 \cdot h \in [ 0, \infty ] $.  When a hermitian metric $(L,h)$ is used to define adjoint norms, it is called a \textbf{singular hermitian metric of the second kind}.  
 A (not necessarily effective) $\QQ$-divisor $D$ on $\xr$ defines a singular hermitian metric of the second kind of the $\QQ$-line bundle $\OO(D)$ by its local equations. It is denoted by $(\OO(D), \eta_{(D)})$. Note that for the purpose of local integrability as in Definition~\ref{wnorm}, a negative coefficient in $D$ only helps since it gives a zero rather than a pole. 

 Let $h_0$ be a smooth hermitian metric of $L$. Let $\phi : X_0 \rrow \RR \cup \{ \pm \infty \} $ be the function defined by $ h = h_0 \cdot e^{-\phi} $. If the function $e^{\phi}$ is bounded above on $X_0$, we say the singular hermitian metric of the second kind $h$ is \textbf{bounded away from zero}. This is independent of the choice of the smooth metric $h_0$. The point of this definition is the following. In general, when $L$ is locally trivialized on an open subset $U$ and $f(s) \in \OO_U$ is the holomorphic function on $U$ corresponding to $s$, there is a measure $d\mu$ on $U$ such that $ \int_U \abs{s}^2 \cdot g = \int_{U} \abs{ f(s) }^2 d \mu $. Given a measure $dV$ associated to a local euclidean volume form on $U$, this $d\mu$ is a $\RR_{\ge 0}$-valued function (say $e^{-\phi}$) times $dV$. If the metric $h$ is bounded away from zero, then by definition $e^{\phi}$ is bounded above on $X_0$, which gives $e^{-\phi} \ge C > 0$ for some $C > 0$. Then up to scaling, $d\mu$ itself can be taken as a measure associated to a local euclidean volume form. We will call such a measure a \textit{volume form}, which we will use in a series of propositions (\ref{bound}), (\ref{montel}) and (\ref{riemann extend}). We need the metric $g$ in Theorem~\ref{main} to be bounded away from zero to apply these propositions. Note that, for example, when a metric $h$ is given by a $\QQ$-divisor $D_1 - D_2$ ($D_1 \neq D_2 \ge 0$), $h$ is not bounded away from zero along the non-effective $-D_2$ since it has zero along $D_2$.

 Now we have the following change of variables property for adjoint norms.

\begin{proposition}[Change of variables formula for adjoint norms]\label{change}  
 Let $ \pi : Z' \rightarrow Z $ be a projective bimeromorphic morphism between complex manifolds.  Suppose that the relative canonical line bundle $ K_{Z'}-\pi^*(K_Z)  $ is expressed as a unique effective exceptional divisor $E$. Let $(M,h)$ be a singular hermitian metric of the second kind on $Z$. Then, for any multi-valued holomorphic section $s \in H^0 ( Z, K_Z + M ) $, we have 
 
 $$ \int_{Z'} |\pi^* s|^2 \cdot h'  =  \int_Z   |s|^2 \cdot h $$
 
\noindent where the singular hermitian line bundle $(\pi^* ( K_Z + M) - K_{Z'}, h')$ is given by the product of $(\pi^* M, \pi^* h)$ and $(\pi^*(K_Z) - K_{Z'}, \eta_{(-E)}) $, the latter given by the divisor $-E$.
 
\end{proposition}

\begin{proof}

 The adjoint norm is not changed when integrated over the complement of a measure zero set, over which $ \pi $ is biholomorphic. See the proof of  Proposition 5.8 of \cite{D94} for the rest of the argument. In this paper, we use this proposition when $\pi$ is given as the restriction over an open subset $ Z \subset \hat{Z} $ of a composition of blowups along smooth centers $ \pi : \hat{Z'} \to \hat{Z} $ where $\hat{Z}$ is a smooth variety.

\end{proof}

\noi For a metric given by an snc divisor, we have the following proposition, which we use in Section 3.

\begin{proposition}\label{snc klt}

 Let $X$ be a complex manifold. Let $L_1$ be a $\QQ$-line bundle, given the singular hermitian metric of the second kind $\eta_{D}$ where $D$ is an snc $\QQ$-divisor and $ D \sim L_1 $. Let $(L_2, g_{L_2} )$ be another $\QQ$-line bundle with a smooth hermitian metric. Then the $\CC$-vector subspace 
 
 $$ \{  s \in \Gm (X, K_X + L_1 + L_2) |  \int_X \abs{s}^2 \cdot \eta_{D} \cdot g_{L_2} < \infty \} $$
 
\noindent is identified with $ \Gm (X, K_X + L_1 + L_2 - \OO(D_1) )$ where $D_1$ is an snc effective divisor whose support is contained in the support of $D$. More precisely, if a prime divisor $S$ appears with the coefficient $\alpha$ in $D$ and $\alpha \ge 1$, then $ [\alpha ] S$ appears in $D_1$ where $[\alpha]$ is the largest integer less than or equal to $\alpha$.

\end{proposition} 
 
\begin{proof}

 The above norm is finite for $s$ if and only if the pair $(X, -\divisor(s) + D)$ is klt, by (3.20) of \cite{Ko97}. It is precisely when $s$ has zero along $ [\alpha ] S$ when a prime divisor $S$ appears with the coefficient $\alpha$ in $D$ and $\alpha \ge 1$. 

\end{proof}

\subsection{Plurisubharmonic functions}\label{upper} 

 We recall definitions and properties of plurisubharmonic functions and quasi-plurisub-harmonic functions. 

\begin{definition}

 A function $\psi : X  \to [ -\infty , \infty ) $ on a topological space $X$ is said to be \textbf{upper semicontinuous} if the sublevel set $ X_c := \{ x \in X \; | \psi (x) < c \} $ is open in $X$ for each $ c \in \RR $.

\end{definition}

\begin{definition}[\cite{D94},(1.4)]\label{psh}

 Let $ U \subset \CC^n $ be an open subset. We say that a function $ \psi : U \to [ -\infty, \infty ) $ is \textbf{plurisubharmonic} if 
 
(a) $\psi$ is upper semicontinuous.

(b) For arbitrary $p \in U$ and $ q \in \CC^n $, we have 

 $$ \psi( p ) \le \frac{1}{2\pi} \int_0^{2\pi} \psi( p + q e^{\ii \theta} ) d\theta  $$

\noindent where the set $ \{ p + q \lambda |  \lambda \in \CC, \abs{\lambda} \le 1 \} \subset U $. 

\end{definition}

\noindent Plurisubharmonic is abbreviated as \textit{psh}. A pullback of a psh function under a holomorphic map is again psh, so it is straightforward to define a $\RR \cup \{ -\infty \}$-valued function on a complex manifold to be \textit{plurisubharmonic} if its pullback on a coordinate chart is psh. 

 The following proposition is application of the sub-mean-value property of a psh function (\cite{D94}, (1.5)) and it will be used in the next section.

\begin{proposition}\label{bound} Let $ W \Subset  U \Subset \CC^n $ be relatively compact open subsets of $\CC^n$ and $d \mu $ a volume form on $U$ such that $U$ has the volume $ V(U) := \int_U 1 d \mu  < \infty $. Then there exists a positive real number $ V \in (0, V(U) ) $ such that for any holomorphic function $f$ on $U$ with the finite norm $ N(U) = \int_U  \abs{f}^2 d \mu < \infty $, we have 
 
 $$ \abs{ f(z) }^2 \le  {N(U)}^{\frac{1}{V} }  $$
 
\noindent for any $ z \in W $. In particular, $ \abs{ f(z) } $ is bounded above on $W$. 

\end{proposition}

\begin{proof}

 Since $ W \Subset U $, we can find a family of open polydiscs $\{ U_z \}_{z \in W}$ of the same volume $ V $ (which is a sufficienty small positive number) such that each $ U_z $ is centered at the point $z$ and contained in $U$. We learned from \cite{F06m}, this way of using the sub-mean-value property with respect to two steps of open subsets, which will be also used later when the current proposition is applied. 

 Since the function $\log \abs{f}$ is plurisubharmonic on $U$, the sub-mean-value property for a plurisubharmonic function and the Jensen inequality for the concavity of logarithm give each of the following two inequalities:  
 
 $$ \log \abs{f(z)}^2 \le \frac{1}{V} \int_{U_z} \log \abs{f}^2 d\mu \le \frac{1}{V} \log ( \int_{U_z} \abs{f}^2 d\mu ) .$$

\noindent Taking the exponential, we get 
 
 $$ \abs{f(x)}^2 \le (  \int_{U_z} \abs{f}^2 d\mu )^{\frac{1}{V}}   \le (  \int_{U} \abs{f}^2 d\mu )^{\frac{1}{V}} .$$

\end{proof}

\noindent We have an immediate corollary: 
 
\begin{corollary}\label{bound-c} 
 
 Let $ W \Subset  U \Subset X $ be relatively compact open subsets of a complex manifold $X$ and suppose that $W$ and $U$ are biholomorphic to connected open subsets $ W' \Subset U' \Subset \CC^n$ (respectively, as the notation suggests). Let $K_X+L$ be an adjoint line bundle on $X$, with $(L,g)$ a singular metric of the second kind, bounded away from zero. Suppose that each of the line bundles $K_X$ and $L$ is trivialized on $U$. Denote the holomorphic function on $U'$ given by a section $ s \in \Gm (U, K_X+L) $ after this trivialization (and $U' \cong U$)  by $f(s) \in \OO (U')$. Let $d \mu$ be the volume form on $U'$ given by 

 $$ \int_U \abs{s}^2 \cdot g = \int_{U'} \abs{ f(s) }^2 d \mu $$
 
\noindent (noting that the choice of $d \mu$ does not depend on the specific section $s$). Then there exists a real number $ V > 0 $ such that for any multi-valued section $ s \in \Gm (U, K_X+L) $ satisfying $ N(U) = \int_U \abs{ s }^2 \cdot g < \infty $, we have  
 
 $$ \abs{ f(z) }^2 \le {N(U)}^{\frac{1}{V} } $$
 
\noindent  for any $ z \in W'$.

\end{corollary}

 Now we turn to discuss families of functions. An important basic property of psh functions is that the pointwise supremum function $\sup(\psi_1, \psi_2)$ is psh if $\psi_1$ and $\psi_2$ are psh. This will be generalized to the supremum over a family of locally uniformly bounded above psh functions. First we need the following general definition:

\begin{definition}[\cite{Ra}, (3.4.1)]\label{usc reg} 

 Let $ \psi : X \to [ -\infty , \infty ) $ be a function which is locally bounded above on a topological space $X$. We define its \textbf{upper  semicontinuous regularization} $ {\psi}^* : X \to [-\infty, \infty ) $ to be the function defined by 
 
 $$ {\psi}^* (x) :=  \limsup_{y \to x} \psi(y)  $$
 
\noindent for each $x \in X$. 
 
\end{definition} 

 A family of functions $ \{ \psi_\al : X \to [ -\infty , \infty ) \}_{ \al \in A}$ is called \textbf{locally uniformly bounded above} if there exists an upper bound for the set $ \{ \psi_\al (y) : \al \in A, x \in Y \} $ for each compact subset $Y \subset X$. The pointwise supremum function $ \psi_A (x) = \sup_{ \al \in A} \psi_\al (x) $ is called the \textbf{upper envelope} of the family. 
 
\begin{proposition}[\cite{Le}, p.26]\label{lelong}

 Let $X$ be a complex manifold and $ \{ \psi_\al \}_{ \al \in A}$ be a locally uniformly bounded above family of psh functions. Then the upper semicontinuous regularization of the upper envelope of the family is also psh.  

\end{proposition}

\noindent For simplicity, we will often use the term {\textquoteleft upper envelope\textquoteright}  to mean its upper semicontinuous regularization. In Chapter 5, we will take the upper envelope of \textit{quasi-plurisubharmonic} functions, defined as follows.

\begin{definition}\label{def qpsh}

 A $\RRR$-valued function $\psi$ on a complex manifold $X$ is \textbf{quasi-plurisubharmonic} (or quasi-psh) if there exists an open covering $ \{ U_i \} ( i \in J) $ of $X$ such that, on each $U_i$, $ \psi $ is the sum $ \psi = v_i + u_i$ of a plurisubharmonic function $v_i$ and a $\RR$-valued $C^\infty$ function $u_i$, both on $U_i$.

\end{definition}

\noindent A family of quasi-psh functions $ \{ \psi_\al \} $ is called \textbf{good} if there exists a common open covering $ \{ U_i \} ( i \in J) $ of $X$ and $\RR$-valued functions $ u_i \in C^{\infty} (U_i) $ such that $ \psi_\al - u_i $ is psh on $U_i$, for any $\al$ and any $i \in J $. An immediate consequence of (\ref{lelong}) is the following

\begin{proposition}\label{qpsh}

 If a good family of quasi-psh functions on a complex manifold $X$ is locally uniformly bounded above, then its upper envelope is also quasi-psh. 
 
\end{proposition}

\subsection{Stein manifolds}\label{Stein}

 In this section, we first introduce Stein manifolds and recall their basic properties from standard references (\cite{GR}, \cite{D97b} Chapter 1). Then following \cite{Siu02}, we introduce an increasing exhaustion sequence of Stein open subsets of $X \setminus H$, a smooth complement of a hyperplane section of a projective variety $X$. We prove Proposition~\ref{montel} and Proposition~\ref{riemann extend} which are used in the proof of Theorem~\ref{main}.

 A Stein space $\Omega$ is a complex analytic space (see Appendix B of \cite{H}) which is characterized by the vanishing of the first cohomology of all coherent analytic sheaves on $\Omega$. (We refer to \cite{GR} for the standard definition of Stein spaces and the proof of this characterization originally due to Serre.)  A Stein manifold is a smooth Stein space. An affine variety (or its associated complex analytic space) is an example of a Stein space. We will use the following fundamental result in the proof of (\ref{dense}). 
 
\begin{proposition}[\textbf{Cartan's Theorem A}, \cite{GR} Chap. 8]\label{CartanA}

 If $\Omega$ is a Stein space and $\mathcal F$ is a coherent analytic sheaf on $\Omega$, then $\mathcal F$ is generated by $ \Gamma ( \Omega, \mathcal F) $. 

\end{proposition}

 On the other hand, we need another characterizing property of Stein manifolds, that is the existence of a smooth strictly psh exhaustion function. First we need the following :

\begin{definition}

 A function $\psi : X  \to [ -\infty , \infty ) $ on a topological space $X$ is said to be an \textit{exhaustion function} if all sublevel sets 
$ X_c := \{ x \in X \; | \psi (x) < c \} , c \in \RR $ are relatively compact (i.e. their closures are compact). 

\end{definition}

\noindent A Stein manifold is \textit{strongly pseudoconvex}, that is, it admits a smooth strictly psh (\cite{D97b}, (5.20))  exhaustion function (\cite{D97b} Chapter 1).
\\

 Now let $X \subset \PP^N$ be a projective variety and $H \subset \PP^N$ a hyperplane such that $ \xs \subset H$. Then $ X \backslash H $ is a smooth affine variety, which is a Stein manifold and therefore admits a smooth strictly psh exhaustion function $\psi$. The sublevel sets of $\psi$ give us an increasing exhaustion sequence of relatively compact Stein open subsets $ \{ \om_t \}_{t \ge 1} $ of the affine variety $ X \setminus H $ : we take $ \om_t = \psi^{-1} ( -\infty, c_t ) $ for an increasing sequence $\{ c_t : t \in \ZZ_{>0} \} $ going to infinity as $ t \to \infty$. By Sard's theorem, we can assume that each $\om_t$ has a smooth boundary $ \partial \om_t$. 

 The proof of our main result Theorem~\ref{main} will use such an increasing sequence of Stein open subsets $ \{ \om_t \}_{t \ge 1} $ with appropriate choice of the hyperplane $H$. $L^2$ methods will give a holomorphic section on each $ \om_t $ and then we will use the following version of the Montel theorem. 

\begin{proposition}\label{montel}

 In the above setting, let $K_X + L$ be an adjoint line bundle~(\ref{adjlb}) on $X \setminus H$ and $(L,g)$ a singular metric of the second kind (on $X \setminus H$) which is bounded away from zero. Suppose that for each $t$, we have a multi-valued section $ s_t \in \Gm ( \om_t, K_X + L) $ with 
 
 $$ \int_{\om_t} \abs{ s_t }^2 \cdot g  \le   C  $$
 
\noindent where $ C > 0 $ is a constant which is independent of $ t \ge 1 $. Then there exists a multi-valued section $ s \in \Gm ( X \backslash H , K_X + L ) $ such that 
 
 $$ \int_{X \backslash H} \abs{ s }^2 \cdot g   \le  C   .$$ 
 
\noindent When $ K_X + L $ is an integral line bundle and each $s_t$ is a holomorphic section, $s$ is also given as a holomorphic section (not just multi-valued holomorphic).

\end{proposition}

\begin{proof}

 We choose and fix a locally finite open covering $\{ W_i \}_{ i \in J } $ of $ X \backslash H $ such that the following hold: 
 
\begin{itemize}

\item 

 For each $ i \in J $, there exists an open subset $ U_i $ such that $ W_i \Subset U_i \Subset X $ and $ W_i \Subset U_i$ are biholomorphic to open sets $ {W_i}' \Subset {U_i}' ( \Subset \CC^n )$ as in Corollary~\ref{bound-c}.

\item
 
 Each of the line bundles $K_X$ and $L$ is trivialized on every $U_i$ ( $ i \in J $ ). (Note that $K_X$ is a line bundle on $ X \backslash H $.) We also have transition functions $ g_{ij} \in \OO(U_i \cap U_j) $ for the line bundle $K_X+L$ on this open covering $ \{ U_i \}_{i \in J} $. 

\item

 For each $ i \in J $, $ W_i \subset \om_{t(i)} $ where $t(i)$ is the smallest positive integer $t$ with $ W_i \subset \om_{t}$.

\item

 Each $ U_i $ is equipped with a volume form $ d \mu_i $ such that the volume $ V(U_i) := \int_{U_i} 1 d \mu_i > 0 $ is finite and also such that 
 
 $$ \int_{U}  \abs{ s }^2 \cdot g   =   \int_{U}  \abs{ f_i }^2 d \mu_i $$
 
 for any subset $ U \subset U_i$, where $f_i$ is the holomorphic function on $U_i$ given by the fixed local trivialization of a section $s$ of $K_X+L$. 

\end{itemize}

\noindent We can indeed choose $ \{ W_i \} $ to be locally finite, inductively on $t$ as follows:  
  For each $ t \ge 1 $, the closure $ \overline { \om_t } $ is a compact subset of $\om_{t+1}$. So one can find a finite number of open sets $w_i \subset \om_{t+1}$ whose union contains $ \om_t \backslash \om_{t-1} $. Take the open intersections $ W_i := w_i \cap \om_t $ and add them to the open covering.

 Now for each $ i \in J $, through the fixed local trivialization of $K_X+L$, the given sections $s_t$ give a sequence of holomorphic functions $ f_{(i,t)} = f_t $ on $U_i$ for $ t \ge t(i) $. Since $ \int_{U_i} \abs{ f_t }^2 d\mu_i = \int_{U_i} \abs{ s_t }^2 \cdot g \le  \int_{\om_{t}} \abs{ s_t }^2 \cdot g          \le  C $, Proposition~\ref{bound} gives the upper bound 
 
 $$ \abs{ f_t }^2 \le C^{\frac{1}{V(U_i)}} .$$
 
\noindent With these bounds, we use the Montel theorem to conclude that (on each $U_i$) there is a subsequence of $\{ f_t = f_{(i,t)} \}_{ t \ge t(i) } $ converging to $ f_i \in \OO (W_i) $. It is possible to choose those limit functions $ f_i \in \OO(W_i) $ for $ i \in J $ such that the collection $ \{ f_i \}_{ i \in J } $ gives an element of $ \Gm ( X \backslash H, K_X+L )$ by the fact that the open cover $\{ W_i \}_{i \in J}$ is locally finite and the following reason: 
 
 For any two different intersecting open sets $W_i$ and $W_j$ ( $ i,j \in J $ ), consider the union $ W_i \cup W_j \subset \om_{t(i,j)} $ where $ t(i,j) = \max(t(i),t(j)) $.  The two sequences of holomorphic functions $ f_{(i,t)} $ on $W_i$ and $f_{(j,t)}$ on $W_j$ come from the same sections $ s_t \in \Gm ( \om_t , K_X + L ) $ for $ t \ge t(i,j) $. Hence $ f_{(i,t)} - g_{ij} f_{(j,t)} = 0 $, $ \forall t \ge t(i,j)$. By the Montel theorem, there is a converging subsequence of $\{ f_{(i,t)} \}$ given by an infinite subset of $t$ indices $T_i \subset \ZZ_{>0}$. Now by the Montel theorem applied on $W_j$, there is a further subsequence ( given by $t$ indices in another infinite subset $ T_j \subset T_i$ ) of $ f_{(i,t)} $ for which the corresponding subsequence of $f_{(j,t)}$ also converges. The last inequality clearly follows.

\end{proof}

 In the proof of Theorem~\ref{main}, the use of the above proposition will be followed by the next proposition, a version of the Riemann extension theorem which extends a bounded holomorphic function across a divisor in a complex analytic space.

\begin{proposition}\label{riemann extend}

 Let $X$ be a normal projective variety and $K_X+L$ an adjoint line bundle on $X$ (\ref{adjlb}). Let $ H_1 \subset X$ be an effective Cartier divisor containing $\xs$. Let $(L,g)$ be a singular hermitian metric of the second kind which is bounded away from zero and whose domain is $X \backslash H_1$.  If a multi-valued section $ s \in \Gm ( X \backslash H_1 , K_X + L) $ on the open complement satisfies 
 
 $$ \int_{X \backslash H_1 } \abs{ s }^2 \cdot g  <  \infty ,$$
 
\noindent then there exists a multi-valued section $ \overline{s} \in \Gm (X, K_X + L) $ such that $ {\overline{s}}|_{X \backslash H_1} = s $. 
\noindent When $ K_X + L $ is an integral line bundle and $s$ is a holomorphic section, $\overline{s}$ is also given as a holomorphic section (not just multi-valued holomorphic). 
  
\end{proposition}

\begin{proof}

  We take and fix a finite collection of open subsets $V_1, \cdots, V_\mu $ of $X$ (not of $X \backslash H_1$ !) satisfying that: For each $ \el = 1, \cdots, \mu $, there is an open subset $ U_\el $ of $ X \backslash H_1 $ such that $ V_\el \backslash H_1 \Subset U_\el $ and that $ U_\el $ is biholomorphic to a connected open subset $ U'_\el \Subset \CC^n $.  
  
 We take an open covering $ \{ V_i \}_{i \in J} $ of $X$ (with $ J \supset \{ 1, \cdots, \mu \} $) such that 
 
1. The line bundle $K_X+L$ is given by transition functions $g_{ij} \in \OO( V_i \cap V_j) $. 

2. For $ i \notin \{ 1, \cdots, \mu \} $, we have $ V_i \cap H_1 = \emptyset $.

\noindent Then the given section $s$ on $X \backslash H_1$ is represented by the collection of holomorphic functions $\{ f_i \}_{i \in J}$ where $f_i$ is holomorphic on $V_i \backslash H_1$ if $ i \in \{ 1, \cdots, \mu \}$ and otherwise, $f_i$ is holomorphic on $V_i$. We apply Proposition~\ref{bound} and the \textit{Riemann Extension Theorem on Normal Complex Spaces} of \cite[p.144]{GR2} to those $f_i$'s on $V_i \backslash H_1$ with $ i \in \{ 1, \cdots, \mu \}$, to obtain $\overline{f_i} \in \OO(V_i)$ extending $f_i$ across $H_1 \cap V_i$. Denoting $\overline{f_i}$ by $f_i$, the new collection $\{ f_i \}_{i \in J}$ satisfies the compatibility  condition $ f_i = g_{ij}f_j $ on $V_i \cap V_j$ since $ f_i - g_{ij}f_j $ is identically zero on $ (V_i \cap V_j) \backslash H_1$. This gives the section $\overline{s}$ we want.

\end{proof}

\quad
\\

\subsection{$\db$ operators on the Hilbert spaces of $(p,q)$ forms}\label{hilbert}

 We begin with the standard functional analytic preliminaries for $L^2$ methods of $\db$ operators, as developed in \cite{Ho65}. Our references also include Chapter 13 of \cite{Ru} and Section 3 of \cite{Siu02}. We start with some of the standard facts about unbounded operators between Hilbert spaces. 
 
 Let $\HH_0$ and $\HH_1$ be two complex Hilbert spaces and let $T$ be an operator $ T: \HH_0 \to \HH_1 $, i.e. a linear map, which may be not necessarily defined on the whole of $\HH_0$. We denote by $\Dom(T)$ the subspace of $\HH_0$ where $T$ is defined.  We define the graph $\mathcal G_T$ of $T$ to be the subspace of $\HH_0 \times \HH_1$ given by $ \mathcal G_T :=  \{ (x, Tx) | \;  x \in \Dom(T) \} \subseteq \HH_0 \times \HH_1 .$  We say that $T$ is a closed operator if $ \mathcal G_T $ is closed in  $\HH_0 \times \HH_1$ and that $T$ is densely defined if $\Dom (T)$ is a dense subspace of $\HH_0$. If $T : \HH_0 \to \HH_1 $ is a closed and densely defined operator, then its adjoint $ T^* : \HH_1 \to \HH_0 $ is defined and it is closed and densely defined. Once we take our Hilbert spaces and operators, the main problem is to solve the equation 

\begin{equation}\label{dbar} 
 T (v) = \alpha 
\end{equation}
 
\noindent  (where $ \alpha \in \HH_1$ is given) for $ v \in \HH_0 $ together with $ \norm{v} \le C $ for a constant $C$. It is helpful to introduce another operator $ S : \HH_1 \to \HH_2$ such that $ ST = 0 $ and we use the following fact (see (3.2), \cite{Siu02}).

\begin{proposition}\label{FA}

 Let $ T: \HH_0 \to \HH_1 $ and $ S : \HH_1 \to \HH_2$ be closed and densely defined operators between Hilbert spaces such that $ ST = 0 $. Suppose that $ S(\alpha) = 0$. There exists a solution $ v \in \HH_1$ of \eqref{dbar} with $ \norm{v} \le C $ if and only if 
 
 $$ C^2 (\norm{ T^* u }^2 + \norm{ S u }^2 )  \ge  \abs{\inner{ u, \alpha }}^2 $$
 
\noindent for all $ u \in \Dom(T^*) \cap \Dom(S) $. 

\end{proposition}

\begin{proof}

 See (3.2) Functional analysis preliminaries, \cite{Siu02}. 

\end{proof}

 After this generality on Hilbert spaces, we introduce the Hilbert space of $L$-valued $(p,q)$ forms on a complex manifold where $L$ is a line bundle. Let $\om$ be a complex manifold with a Hermitian metric $\xi$ and $(L,g)$ a singular hermitian $\QQ$-line bundle of the first kind on $\om$. Let $dV$ denote the volume form defined by $\xi$. 
 
\noindent Let $V \subset \om $ be an open neighborhood of a point in $\om$ with an orthonormal coframe $ \og_1, \cdots, \og_n $ of type $(1,0)$. We can also assume that there exists $ \theta_V $, a local frame of $L$ over $V$ and put $e^{-\varphi} = g( \theta_V, \theta_V ) $.    
 
 Following \cite[p.121]{Ho65}, we define $ \LL^2_{(p,q)} ( \om, L, g ) $ as the Hilbert space completion of all smooth $L$-valued $(p,q)$ forms square integrable with respect to the singular metric $(L,g)$ in the sense that the following norm is finite :  
 
 $$ \norm{u}^2 := \int_\om \abs{u}^2_g dV  <  \infty $$
 
\noindent where $ \abs{u}^2_g $ is well defined when we locally define it on each open subset $V \subset \om$ to be 
 
 $$  \abs{u}^2_g :=  \frac{1}{p! q!} \sum_{ \abs{I} = p, \abs{J} = Q}  \abs{ u_{I,J} }^2 \cdot e^{-\varphi} $$
 
\noindent when the expression of $u$ on $V$ is given by $  u = \sum  u_{I,J} \theta_V \otimes \omega^I \wedge \mg^J $. Similarly, the pointwise inner product $ \inp{ u,v }_g $ and its integral $ \inner{ u,v } = \int_\om  \inp{u,v}_g dV $ are defined. 
\qa
\\

 From now on, we take $ p = n $ and $ q = 0,1,2 $. The complex manifold $\om$ will always be a relatively compact Stein open subset in a smooth affine variety $X$ and $\xi$ a K\"ahler metric on $X$.  In this setting of $L^2$ methods for the $\db$ operator, our operators between the Hilbert spaces $ \LL^2_{(n,q)} ( \om, L, g ) $ are taken as $ T = \db, S = \db $ or $ T = \db ( \sqrt{\eta_1} \; \cdot )  ,  S = (\sqrt{\eta_2}) \db (\cdot) $ where $ \eta_1, \eta_2 \ge 0 $ are functions on $\om$ to be multiplied to $L$-valued $(n,q)$ forms. We note that the composition $ S  T = 0$ and $ \Dom (T^*) = \Dom ( \db^* ),  \Dom (S) = \Dom (\db) $ in either case.

 In the context of using Proposition \ref{FA} with these $T$ and $S$, there is a fundamental result (Proposition~\ref{BK0}) giving a lower bound of $\norm{ T^* u }^2 + \norm{ S u }^2 $ for $ u \in \HH_1 = \LL^2_{(n,1)} ( \om, L, g ) $. To state it,  first we need to define (for a $C^2$ function $\psi$ on $\om$ ),

\begin{equation}\label{defn}
   ( \iddb \psi)(u,u)_g  := \inp{ [ \iddb \psi, \Lambda]u, u}_g = \inp{ (\iddb \psi)(\Lambda u), u}_g   
\end{equation}

\noindent where $\Lambda$ is the adjoint of the operator $ \omega_\xi \wedge \cdot $ given by the K\"ahler form $\omega_\xi$ of $\xi$ and the inner product $\inp{ \; , \; }$ is taken pointwise as that of $\LL^2_{(n,q)} ( \om, L, g )$. Locally we have 
 
 $$ ( \iddb \psi)(u,u)_g  = \sum \frac{ \de^2 \psi }{ \de \mg_i \de \og_j } u_i \overline{u_j} \cdot e^{-\varphi}  $$

\noindent when $ u = \sum_{j=1}^n  u_{j} \theta_V \otimes \omega^I \wedge \mg^j  $ on $V \subset \om$ as above. (The first order linear differential operators $ \frac{ \de}{\de \mg_i} $ and $ \frac{ \de}{\de \og_j} $ are defined by the relation $ d\psi = \sum_1^n \frac{ \de \psi }{\de \mg_i } \mg_i  + \sum_1^n \frac{ \de \psi}{\de \og_i } \og_i $ as in \cite[p.122]{Ho65}) In the place of the $(1,1)$ form $ \iddb \psi $ in \eqref{defn}, we can also put any closed real semipositive $(1,1)$ form $\Pi$, for which we can find $\psi$ locally such that $ \Pi = \iddb \psi$. 
\\
\qa 
 
 Now going back to the modified $\db$ operators, $ T = \db ( \sqrt{\eta_1} \; \cdot )  ,  S = (\sqrt{\eta_2}) \db (\cdot) $, we determine our functions $\eta_1, \eta_2$ to work with, following McNeal and Varolin \cite{MV}, \cite{Var}. 
 Let $\lambda$ be a $C^2$ function defined on $\om$. Following \cite{MV}, we first consider an auxiliary function $ r(x) = 2 - x + \log (2 e^{x-1} - 1) $ for $ x \ge 1$. Note that $ \displaystyle r'(x) = \frac{1}{ 2e^{x-1} - 1 } \in (0,1) $ for $x \ge 1 $.   We define functions  
 
 $$\eta = \ld + r ( \ld) \qa \text{ and } \qa \gamma = \displaystyle \frac{(1+r'(\ld))^2}{-r''(\ld)} .$$   
 
\noindent It is easy to see that $\ld + r(\ld) \le 1 + \log 2 + \ld$ and $ \gamma = 2 e^{\ld-1}$. From Section 3.2 of \cite{Var}, we have
 
\begin{equation}\label{eta}
 \eta \ge 1 + r' (\lambda)  > 1 
\end{equation}

\noindent and 

\begin{equation}\label{eta2}
  - \iddb \eta - \frac{\sqrt{-1}}{\gamma} \partial \eta \wedge \overline{\partial} \eta = ( 1 + r'(\lambda)) ( -\iddb \lambda). 
\end{equation}

\noindent We put $ T := \db ( (\sqrt{ \eta + \gamma }) \; \cdot ) $, composition of multiplication by the function $ \sqrt{ \eta + \gamma }$ first and then taking $\db$. Similarly, we let $ S := (\sqrt{\eta}) \db ( \cdot) $.

\begin{proposition}[Twisted Basic Estimate : Ohsawa-Takegoshi, Siu, McNeal-Varolin]\label{BK0}

 Let $(\om, \xi)$ be a relatively compact Stein open subset of a Stein manifold, with the smooth boundary $ \partial \om$. Let $(L,g)$ be a  smooth hermitian line bundle with the curvature $(1,1)$ form $\ii \Theta_g (L) $. For the operators $T$ and $S$ defined above in terms of a \textbf{$C^2$} function $\lambda$, we have 
 
\begin{align*}
\norm{T^* u}^2 + \Vert S u \Vert^2 &\ge \int_{\om} ( \eta \sqrt{-1} \Theta_{g}(L) - \sqrt{-1} \partial \overline{\partial} \eta - \frac{1}{\gamma} \sqrt{-1} \partial \eta \wedge \overline{\partial} \eta ) (u,u)_g dV \\
& = \int_{\om} ( \eta \sqrt{-1} \Theta_{g}(L) + ( 1 + r'(\lambda)) ( -\iddb \lambda) ) (u,u)_g dV
\end{align*}

\noindent for any $ u \in \Dom(T^*) \cap \Dom(S) \subset \LL^2_{(n,1)} ( \om, L, g ) $. 
 
\end{proposition}

\begin{proof}

 See Proposition 3.4 \cite{Siu02} and Section 2.1 \cite{MV}.  \eqref{eta2} was used for the equality. 

\end{proof}

\section{Kawamata metric on a log-canonical center}\label{subadjunction}

\subsection{A refined log-resolution and the Kawamata metric}\label{center}

 In this section, we first recall the notion of a log-canonical center following \cite{Ka97}, \cite{Ka98}, \cite{Ko97} and \cite{Ko05}. Then we define the Kawamata metric on an lc center (Definition~\ref{kawamata}) and prove its main property Theorem~\ref{kltklt}, which is crucial in the proof of Theorem~\ref{main}.

 Let $X$ be a normal variety and $D$ a (not necessarily effective) Weil $\QQ$-divisor such that the sum of the two Weil divisors $K_X + D$ is $\QQ$-Cartier. By Hironaka's theorem, there exist log-resolutions $ f: X' \to X $ of the pair $(X,D)$. Then as a $\QQ$-line bundle, we have the equality $ K_{X'} = f^* (K_X + D) - D' - \Delta $ where $D'$ is the birational transform of $D$ under $f$ and $\Delta$ a combination of exceptional divisors. We say the pair $(X,D)$ is \textbf{klt} (or Kawamata log-terminal) if there exists such $f$ with each prime divisor in $- D' - \Delta$ has its coefficient (called the \textit{discrepancy}) greater than $-1$. We say $(X,D)$ is \textbf{lc} (or log-canonical) if each discrepancy is greater than or equal to $-1$.  These are well-defined, independent of the choice of $f$. 
 
 Let $(X,D)$ be an lc pair. A \textbf{log-canonical center} (or an \textbf{lc center}) of $(X,D)$ is an irreducible subvariety $Z \subset X$ which is the image of an exceptional divisor with its discrepancy equal to $-1$ on a log-resolution of the pair $(X,D)$. If $(X,D)$ is lc but not klt, then it has at least one and at most a finite number of lc centers on $X$. 
 
 If $Z_1$ is an lc center and there is no other lc center $Z_2$ such that $ Z_2 \supsetneq Z_1 $, then $Z_1$ is called a \textbf{maximal lc center} following \cite{T06}. If $Z_1$ and $Z_2$ are lc centers of $(X,D)$, then an irreducible component of $Z_1 \cap Z_2$ is also an lc  center(\cite{Ka98}). So for each point $x \in X$ such that $(X,D)$ is not klt at $x$, there is an lc center $Z \ni x$ that is the unique minimal lc center with respect to set-theoretic inclusion. In such a case, we say $Z$ is a minimal lc center at $x$. We call $Z_1$ a \textbf{minimal lc center} of $(X,D)$ if $Z_1$ is minimal at every point $x \in Z_1$. A minimal lc center $Z$ of $(X,D)$ is a normal subvariety(\cite{Ka97}, Theorem1.6).

 A maximal lc center may contain more lc centers as irreducible closed subsets, in particular minimal lc centers. One can perturb $D$, that is, replace it by $(1-\varep_1)D + \varep_2 H$ where $H$ is an ample divisor and $1 \gg \varep_1, \varep_2 > 0$ to make a given lc center into a maximal lc center of the perturbed pair $(X, (1-\varep_1)D + \varep_2 H)$. An lc center may possibly be both maximal and minimal, in which case any other lc center of $(X,D)$ is disjoint from $Z$. If a maximal and minimal lc center moreover satisfies that it has exactly one exceptional divisor with the discrepancy $-1$, it is called \textbf{exceptional minimal} (or \textit{exceptional} as in \cite{Ko05}). A minimal lc center of $(X,D)$ can be made into an exceptional minimal one by perturbing $D$.  
\\
\qa

\label{reflog}
 
 After these basic notions, we introduce a \textbf{refined log-resolution} of an lc pair with respect to an lc center, following \cite{Ka98} and \cite{Ko05}. We will use it to define the Kawamata metric (Definition~\ref{kawamata}). A refined log-resolution is a log-resolution where the morphism from an exceptional divisor $E$ to an lc center $Z$ is replaced by one from $E$ to $Z'$ ($Z'$ is birational over $Z$) which satisfies better properties in terms of snc divisors. 
 
\noindent More precisely, let $Z$ be an (not necessarily minimal) lc center of an lc pair $(X,D)$ and $E$ an exceptional divisor with discrepancy $-1$ over $Z$. We choose a log-resolution $ f : X' \rrow X $ of $(X,D)$ such that the following holds: 
 
\noindent If we write the relative canonical divisor on $X'$ as 

\begin{equation}\label{rel} 
 K_{X'} = f^* (K_X + D) - E - D' - \Delta 
\end{equation} 
 
\noindent  (where $D'$ is the birational transform of $D$ and $\Delta$ a combination of exceptional divisors whose coefficients are less than or equal to $1$) and put 
 
 $$ R_1 := ( D' + \Delta )|_E     ,$$     then there exists a smooth variety $Z'$, a morphism $ f_E : E \rrow Z'$, a birational morphism $ \pi : Z' \rrow Z $ and a reduced (i.e. all nonzero coefficients equal to $1$) snc divisor $Q_1$ on $Z'$, satisfying the standard snc conditions (\ref{ssnc}) when we take $ f = f_E, X' = X = E, Y' = Y = Z', R = R_1 $ and $ Q = Q_1 $.

$$ 
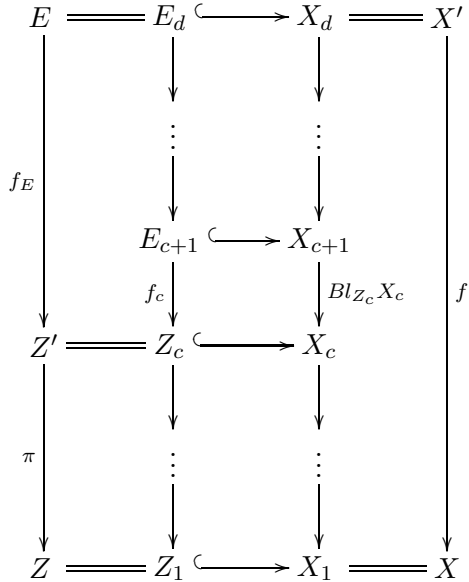
\xymatrix{ E\; \ar@{=}[r] \ar[ddd]_{f_E} &    E_d\; \ar[d] \ar@{^{(}->}[r]          & X_d\; \ar@{=}[r] \ar[d] &    X' \ar[ddddd]^{f} \\
                                            & \vdots \ar[d]                            & \vdots \ar[d]    \\
                                            & E_{c+1} \; \ar[d]_{f_c}     \ar@{^{(}->}[r]    & X_{c+1} \ar[d]^{  Bl_{Z_c} X_c}  & \\
              Z' \ar@{=}[r] \ar[dd]_{\pi}   & Z_c\; \ar[d] \ar@{^{(}->}[r]             & X_c \ar[d]        \\
                                            & \vdots \ar[d]                            & \vdots \ar[d]    \\
              Z\; \ar@{=}[r]                &  Z_1 \;  \ar@{^{(}->}[r]                 & X_1\; \ar@{=}[r] &    X           }$$

\noindent  Then we apply Proposition~\ref{Ka98} for a projective morphism satisfying the standard snc conditions, to the morphism $f_E$ from the exceptional divisor $E$ down to $Z'$.  It follows that we can write
    
\begin{equation}\label{QR1}    
 K_E + R_1 = {f_E}^* ( K_{Z'} + J + Q(R_1) ) 
\end{equation}

\noindent where $J$ is a $\QQ$-line bundle and $Q(R_1)$ is the unique smallest $\QQ$-divisor supported on $Q_1$ among those satisfying 

\begin{equation}\label{QR}
 (R_1)_v + {f_E}^* ( Q_1 - Q(R_1)) \le   \red ({f_E}^* Q_1) .
\end{equation}

\noindent Note that $Q(R_1)$ is not necessarily effective.  Fix a smooth hermitian metric $\gamma_J$ of the $\QQ$-line bundle $J$.  We do not need any curvature property of $\gamma_J$ or any property of the line bundle $J$. Let $ \eta_{Q(R_1)} $ be the singular metric associated to the divisor $Q(R_1)$. The product $\gamma_J \cdot \eta_{Q(R_1)}$ gives a singular metric for the line bundle $M'$ which is defined by $ K_{Z'} + M' = \pi^* (K_X+L)|_Z $ on $ \pi^{-1} (Z_{\reg}) \subset Z'$, when we denote the $\QQ$-line bundle $\OO(K_X + D)$ by $K_X + L$. 
 
 Let $Z_0 \subset Z_{\reg}$ be the largest open subset over which $\pi$ is an isomorphism. There is a $\QQ$-line bundle $M$ on $Z_0$ such that $K_{Z'} + M' = \pi^* (K_{Z_0} + M)$. On $Z_0$, we can identify $M'$ and $M$ and define the following metric for $M$ using $Q(R_1)$ in \eqref{QR1}.

\begin{definition}\label{kawamata}
 
 Let $Z$ be an lc center of an lc pair $(X,D)$ with $D \ge 0$. Choosing a refined log-resolution for $Z$ as above and identifying  $ M' \cong M $, there is a singular hermitian metric $h$ of $M$ of the second kind (whose domain is $Z_0$) given by $ (M,h) \cong (M', \gamma_J \cdot \eta_{Q(R_1)} )$. We call $(M,h)$ a \textbf{Kawamata metric} on the lc center $Z$ of the pair $(X,D)$. 
\end{definition} 
 
\noindent Note that a Kawamata metric depends on the choice of a log-resolution, the choice of $\gamma_J$ and so on, which does not matter to our use of it. We use it to define the adjoint norm of a given section of $(K_X+L)|_Z$ to be extended from $Z$, in the $L^2$ extension Theorem~\ref{main}.

 The key property of a Kawamata metric is the next theorem, which shows that the adjoint norm in terms of a Kawamata metric is precisely what we need in formulating Theorem~\ref{main}. 
 
\begin{theorem}\label{kltklt}

 Let $Z \subset X$, $K_X+L$ and $h$ as in Definition~\ref{kawamata}. Let $V \subset \xr$ be a connected open Stein subset such that $ \emptyset \neq V \cap Z \subset Z_0$. If given any singular hermitian line bundle $(B,b)$ of the first kind on $V$ and a  section $ \st \in \Gamma (V, K_X +L+B)$ with its restriction $ \st|_Z $ on $Z$ satisfying 
 
 $$ \int_{ V \cap Z} \abs{ \st|_Z }^2 \cdot h \cdot b|_Z < \infty   ,$$
 
\noindent then the pullback $f^* \st \in \Gm (f^{-1} (V), f^*(K_X +L+B))$ satisfies 
 
 $$ \int_{ f^{-1} (V) }  \abs{ f^* \st }^2 \cdot \eta_{(D' + \Delta)} \cdot \gamma_{\OO(E)} \cdot f^* b   < \infty $$ 
 
\noindent where $\eta_{(D' + \Delta)}$ is the singular metric associated to the divisor $D' + \Delta$ in \eqref{rel} and $\gamma_{\OO(E)}$ is any smooth hermitian metric of $\OO(E)$.

\end{theorem}

\begin{proof}

 The idea of the proof is to use the relation between klt divisors and finiteness of adjoint norms (as in \cite{Ko97}, (3.20)), especially for snc divisors. 

 Let $L'$ be the line bundle on $X'$ defined by the relation $ K_{X'} + L' = f^* (K_X + L)$. We define $\CC$-vector subspaces $ \Gamma_1  \subset \Gamma ( V , K_X +L +B ) $ and $\Gamma_2  \subset \Gamma ( f^{-1} (V), K_{X'} + L' + f^* B ) $  by 
 
  $$ \Gamma_1 := \{   \st \in \Gamma ( V , K_X +L +B )      \vert     \int_{ V \cap Z} \abs{ \st|_Z }^2 \cdot h \cdot b|_Z  < \infty       \}     $$ and

 $$ \Gamma_2 := \{ \sigma \in \Gamma ( f^{-1} (V), K_{X'} + L' + f^* B ) \vert   \int_{f^{-1} (V)}   \abs{ \sigma}^2 \cdot  \eta_{(D' + \Delta)} \cdot \gamma_{\OO(E)} \cdot f^* b       < \infty  \}  .$$
 
\noindent We need to show that $ f^* \Gamma_1 \subset \Gamma_2 $ as subspaces of $\Gamma ( f^{-1} (V), K_{X'} + L' + f^* B )$. We will reduce this to showing the inclusion only of a dense subset of $ f^* \Gamma_1$ in a topology to be specified. 
 
 First using Demailly's \textit{approximation of psh functions by logarithms of holomorphic functions }(\cite{D97}, Section 6) on $V$, we can assume that the singular metric of the first kind $b$ is given by an effective $\QQ$-divisor $\beta$ (having $\JJ(\beta) = \JJ(b)$). The divisor $\beta$ itself is not necessarily snc. We replace the log-resolution $f$ by another $f$, having additional intermediate blow-ups so that it factors through a log-resolution $f_1: V' \to V$ of the pair $(V, D + \beta)$. We take this new log-resolution in such a way that
 
\begin{enumerate}

\item 
  The divisor $ {f_1}^* \beta $ is snc. 

\item  
  The restriction of $ {f_1}^* \beta $ to the inverse image of $V \cap Z$ (a subvariety in $V'$) makes an snc divisor when it is added to the inverse image of $Q(R_1)$ coming from $ {\pi}^{-1} (V \cap Z)$.
  
\item
  The pullback $f^* \beta $ makes an snc divisor when it is added to $ E+ D' + \Delta$ on $f^{-1} (V)$. ( This last condition is included in the fact that $f$ is a log-resolution of the pair $(V, D+ \beta)$.)
  
\end{enumerate}  
  
\noindent  In the rest of the proof, we work with these snc divisors on $ {f_1}^{-1} \pi^{-1} (V \cap Z) \subset V' $ instead of on $ \pi^{-1} (V \cap Z) \cong V \cap Z \subset V$. But for simplicity in notation, we will write under the notational assumption that the snc conditions as in 1),2) and 3) are being achieved at the level before going up by $f_1$. 
 
 Reduction of showing $ f^* \Gamma_1 \subset \Gamma_2 $ to a dense subset of $ f^* \Gamma_1$ is given by the following lemma. First, we use the fact that the space of global sections $\Gm (V, \mathcal{F})$ is a topological vector space as a Fr\'echet space (\cite{Ru}, \cite{D97b}) for a coherent sheaf $\mathcal{F}$ on a complex analytic space $V$. We always use this topology for $\CC$-vector spaces appearing as a subspace of some $\Gm (V, \mathcal{F})$.  
 
\begin{lemma}\label{dense}

 The following subset of $ \Gamma_1 $ is dense in $\Gamma_1$ : 
 
\begin{align*} 
 \{ \st \in \Gamma_1   \vert  \;\text{The divisor}\;\; & \pi^* \divisor( \st|_Z ) + Q(R_1) + \pi^*(\beta|_Z)  \;\text{ is snc on }\;  \pi^{-1} (V \cap Z) \subset Z' \;\text{and}\\  \;\text{the divisor}\;\;  &  f^* \divisor(\st) + E + D' + \Delta  \;\;\text{is snc on}\; f^{-1} (V) \subset X'    \;   \}.
\end{align*}

\end{lemma}

\begin{proof}

 Note that $\pi : \pi^{-1} (V \cap Z) \to V \cap Z $ is isomorphism since $ V \cap Z \subset Z_0$. We view $V_1 := \pi^{-1} (V \cap Z)$ as a subvariety of $V$ under this isomorphism.
 
 The conclusion will follow from Proposition~\ref{CartanA} and Corollary~\ref{c-bertini}, once we have that $\Gamma_1$ (being a subspace of $\subset \Gm(V, K_X +L+B)$ ) is itself given as the space of global sections of an invertible subsheaf of $K_X +L+B$.  For the restriction $ \Gamma_1 |_{V_1}$, this is given by Proposition~\ref{snc klt}. It then follows for $\Gamma_1$ by extending the line bundle from $V_1$ to $V$ (which is given by the associated line bundle of a divisor extended from $V_1$ to $V$). Since $V$ is Stein, there is only one extension as a line bundle.          

\end{proof}
 
\noindent Using Lemma~\ref{dense}, it suffices to show that $ f^*\st \in \Gamma_2  $ when the divisor $ \pi^* \divisor( \st|_Z ) + Q(R_1) + \pi^*(\beta|_Z) $ is snc on $ \pi^{-1} (V \cap Z) \subset Z'$ and $ f^* \divisor(\st) + E + D' + \Delta $ is snc on $f^{-1} (V) \subset X'$.  In that case, define $ s := \st|_Z $ and define $\QQ$-divisors

\begin{align*}  
  R_2 &:= R_1 - f^* \divisor (s) + f^* (\beta|_Z) \\ 
  Q_2 &:= Q_1 + \red (\pi^* ( \divisor (s) + (\beta|_Z) )) \text{ \qquad  and} \\ 
  \Theta &:= Q(R_1) - \pi^* \divisor (s) +  \pi^*(\beta|_Z) . 
\end{align*} 
 
\noindent Then we have $ (R_2)_h = (R_1)_h $ and $ (R_2)_v = (R_1)_v - f^* \divisor (s) + f^* (\beta|_Z)$. The following shows that $Q(R_2) \le \Theta$ (see the general definition of $ R \mapsto Q(R)$ in Proposition~\ref{Ka98}).

\begin{align*}
 (R_2)_v + {f_E}^* ( Q_2 - \Theta ) &= (R_1)_v - f^* \divisor(s) + f^* (\beta|_Z)  + {f_E}^* ( Q_1 - Q(R_1) ) + \\
  &  \qquad  \qquad {f_E}^* ( \red (\pi^* \divisor(s) + \pi^* (\beta|_Z)  )) + {f_E}^* \pi^* \divisor(s)  - {f_E}^* \pi^* (\beta|_Z)     \\
                                    &\le \red ( {f_E}^* Q_1) + {f_E}^* ( \red ( \pi^* \divisor (s)  + \pi^* (\beta|_Z) ))   \\
                                    &= \red ( {f_E}^* Q_2) 
\end{align*}

\noindent where the inequality follows from \eqref{QR} and the fact that $f = f_E \circ \pi $ and the last equality from the fact that the divisor ${f_E}^* ( \red (\pi^* \divisor(s)  + \pi^* (\beta|_Z)) )$ is already reduced. 
 
 Now the finiteness of the norm with respect to the Kawamata metric $$ \displaystyle \int_{ V \cap Z} \abs{ s }^2 \cdot h \cdot b|_Z < \infty $$ implies that the pair $ (Z', \Theta = Q(R_1) - \pi^* \divisor (s) +  \pi^*(\beta|_Z) ) $ is klt. Since $ Q(R_2) \le \Theta $, the pair $(Z', Q(R_2) )$ is also klt, which implies that $ ( E, R_2 )$ is klt by Proposition~\ref{Ka98}. Note that $R_2$ on $ f^{-1} (V) \subset X'$ is the snc divisor $ R_2 = ( D' + \Delta - f^* \divisor (\st) + f^* (\beta) )|_E $. The kltness of an snc divisor is simply characterized by its coefficients \cite[(3.19.3)]{Ko97}, so the pair $(X', D' + \Delta - f^* \divisor (\st) + f^* (\beta))$ is klt by \cite[(7.4)]{Ko97} (or also by \cite[(7.2.1.2)]{Ko97}). Thus we have $$ \int_{ f^{-1} (V) }  \abs{ f^* \st }^2 \cdot \eta_{(D' + \Delta)} \cdot \gamma_{\OO(E)} \cdot f^* b   < \infty  .$$  Theorem~\ref{kltklt} is proved.

\end{proof}

\subsection{Appendix}  

 We first give the following definition of a property of a projective morphism $f$ between complex analytic spaces given as analytic open subsets of varieties.

\begin{definition}(Standard snc conditions)(\cite[(8.3.6)]{Ko05})\label{ssnc}

 We say that $ f: X \rrow Y$, a divisor $ R \subset X $ and a reduced divisor $ Q \subset Y $ satisfy the \textbf{standard snc conditions} if the following hold : 
 
 (1) $f$ is the restriction of a surjective projective morphism $ f: X' \rrow Y'$ between smooth varieties on a connected open (in the analytic topology) subset $ Y \subseteq Y'$,
 
 (2) $ R + f^* Q $ and $Q$ are snc divisors,
 
 (3) $f$ is smooth over $ Y \setminus Q $, 
 
 (4) $R_v$ is supported in $ f^{-1} (Q) $, and 
 
 (5) $R_h$ is a relative snc divisor \footnote{ We note that according to \cite{Ka98}, (8.3.6.4) of \cite{Ko05} should read that $R_h$ is a relative snc divisor instead of $R$.} over $ Y \setminus Q $, that is:
 
\noindent  for each closed point $x$ of $X$, there exists an open neighborhood $U$ and $ u_1, \cdots, u_k \in \OO_{X,x} $ inducing a regular system of parameters on $ f^{-1} (f(x)) $ at $x$ where $ k = \dim_x  f^{-1} (f(x)) $ such that $ R_h \cap U = \{ u_1 \cdots u_l = 0 \} $ for some $l$ such that $ 0 \le l \le k$  \;  (\cite{F99}).

\end{definition}

\begin{proposition}\label{Ka98} 

 Let $ f: X \to Y $ and $R , Q$ satisfy the standard snc conditions (Definition~\ref{ssnc}). Assume that the $\QQ$-line bundle $\OO(K_X + R)$ is the pullback under $f$ of a $\QQ$-line bundle on $Y$.  Let $ R = R_h + R_v$ be the horizontal and the vertical parts of $R$. Assume that $ R_h \ge 0 $ and that each coefficient of a component of $R_h$ is less than $1$.

 Then there is the unique smallest $\QQ$-divisor supported on $Q$ among those satisfying 

$$ R_v + {f}^* ( Q - Q(R) ) \le   \red ({f}^* Q) $$    and we denote the divisor by $Q(R)$. Moreover, the pair $(Y, Q(R))$ is klt if and only if $(X, R)$ is klt.

\end{proposition}

\begin{proof}

See Theorem 8.3.7 of \cite{Ko05}.

\end{proof}

 On the other hand, the following is the analogue of the Bertini theorem on a complex manifold and its corollary, which we used in the proof of Theorem~\ref{kltklt}.
 
\begin{proposition}\label{bertini}

 Let $W$ be a complex manifold and $M$ a holomorphic line bundle on $W$. Suppose that a vector subspace $ \Gamma \subset \Gm (W, M) $ generates the line bundle $M$. Then the subset of smooth divisors in the topological vector space $\Gamma$ is dense. 

\end{proposition}

\begin{proof} 

 As in the statement, we will often identify a section in $\Gamma$ with the divisor defined by the section. 
 We will show how the argument in the proof of the original Bertini theorem in \cite[pp.137-138]{GH} is adapted in our situation. Suppose that the subset of smooth divisors in $\Gamma$ is not dense. (*): Then there exists an open subset $ f + \om $ of the topological vector space $\Gamma$, where $ f \in \Gamma$ is an element and $\om$ is an open neighborhood of the origin, such that each divisor in $ f + \om $ has a singular point. 
  
 By definition of a topological vector space, for any $x \in \Gamma$, the scalar multiplication map $ \CC \to \Gamma$ sending $\alpha$ to $\alpha x$ is continuous. Therefore the set $ \{ \alpha \in \CC | \alpha x \in \om \} $ is an open set in $\CC$ containing $0$. It follows that any $x \in \Gamma$ has some scalar multiple $ \alpha x \in \om $ for some $ \alpha \neq 0$. Now define a set $V$ of points on $W$ as 
 
 $$ V := \{ P \in W | \text{there exists a divisor} \;  D_P \in \Gamma \; \text{such that} \; P \; \text{is a singular point of} \; D_P \} .$$
 
\noindent For each finite dimensional subspace $\Gamma_1$ of $\Gamma$, the subset of $V$ given by singular points of divisors in $\Gamma_1$ is an analytic subset of $W$, as is explained in \cite[p.138]{GH} for the case of a pencil. So $V$ is the countable union of analytic subsets of $W$. 
 
 Since $\Gamma$ generates the line bundle $M$, there exists a section $ g \in \Gamma $ which is nonzero at (at least) one singular point of $\divisor (f)$. (By definition of $ f + \om $, $\divisor (f)$ has a singular point.)  Consider the linear system $\Gamma_{f,g}$ generated by $f$ and $g$.  Let $V_1 \subset W $ be the analytic subset which is precisely composed of singular points of divisors in $\Gamma_{f,g}$. Let $B$ be the base locus of $\Gamma_{f,g}$, that is, the analytic subset of $W$ given by $ f = g = 0 $. By the above choice of $g$, we have $ V_1 \subsetneq \divisor(g) $. By the calculation with local equations of $f$ and $g$ in \cite[pp.137-138]{GH}, the ratio function $\frac{f}{g}$ is constant on every connected component of $V_1 - B$. 
 
 Considering those divisors $ f + \lambda g \in f + \om$ arising from (*), we get contradiction since $V_1 - B$ meets infinitely many divisors given by those $ f + \lambda g$'s.  

\end{proof}

\begin{corollary}\label{c-bertini}

 Let $W$ be a complex manifold and $\sum S_i$ a reduced snc divisor on $W$. Let $M$ be a line bundle on $W$ which is generated by its global sections. Then the subset in $\Gm (W, M)$ of those sections $s$ having $\divisor (s) + \sum S_i$ snc, is dense. 
 
\end{corollary}

\begin{proof}

 This immediately follows from Proposition~\ref{bertini}, as in \cite[(9.1.9)]{L}. Note that when a line bundle $M$ is generated by $\Gm (W, M)$, the restricted line bundle $M|_S$ to a submanifold $S \subset W$ is not only generated by $\Gm (S, M|_S)$, but also generated by the restricted sections $ (\Gm (W, M))|_S$. 
 
\end{proof}

\section{$L^2$ extension}

 In this section, we prove our main result Theorem~\ref{main}.

\subsection{Statement of the main theorem}

 First we introduce the following $L^2$ extension theorem of Siu~\cite{Siu02} which he used in his proof of invariance of plurigenera for smooth projective varieties not necessarily of general type. 
 
\begin{theorem}[Siu, \cite{Siu02}]\label{OTS}

 Let $ \pi: \mathcal{X} \to \Delta$ be a smooth family of projective varieties over the unit disk $ \Delta \subset \CC $. Let $\mathcal{X}_0$ be the fiber $\pi^{-1} (0)$ over the point $ 0 \in \Delta$, which is a smooth projective variety. Let $(B, b)$ be any line bundle having a singular metric with nonnegative curvature current on $\XX$ and let $K_{\XX}$ be the canonical line bundle of $\XX$. If  $ s \in H^0( K_{\mathcal{X}_0} + B|_{\mathcal{X}_0})$ is a holomorphic section with
 
 $$ \int_{\mathcal{X}_0} \abs{s}^2 \cdot b|_{\mathcal{X}_0}  < \infty   ,$$ 
 
\noindent then it can be extended to a holomorphic section $ \st \in H^0( K_{\XX} + B)$ (that is, $ \st|_{\mathcal{X}_0} = s $) such that

\begin{equation}\label{extend1}
 \int_{\mathcal{X}} \abs{\st}^2 \cdot b  \le  C  \int_{\mathcal{X}_0} \abs{s}^2 \cdot b|_{X_0} ,
\end{equation}

\noindent where $C$ is a universal constant. 

\end{theorem} 
    
 In the proof of Theorem~\ref{OTS}, an important role is played by a real-valued function of the type $\log ( \abs{\omega}^2 + \ep^2)$ where $\omega$ is the global equation for the divisor $\mathcal{X}_0$ in $\XX$ and $\ep$ is an auxiliary variable (for which we will take $\ep \to 0$). In our setting of $Z \subset X$, a subvariety of codimension $k$ of a projective variety, we need a similar function replacing $\abs{\omega}^2$ by $ \abs{\omega_1}^2 + \cdots + \abs{\omega_k}^2$ where $\omega_1 = \cdots = \omega_k = 0 $ give the equations for $Z$ in $X$. Of course, we cannot have one set of such global equations.  Instead, we only need the existence of a globally defined function $\ld$ which satisfies conditions \eqref{condition1} and \eqref{condition2} with respect to local equations of $Z$. Such a function $\ld$ can be constructed in the following setting of a maximal log-canonical center which gives our main result Theorem~\ref{main}.

\qa
\\

\noi  Let $X$ be a normal projective variety and $D \ge 0$ an effective $\QQ$-divisor such that the pair $(X,D)$ is log-canonical.  Let $Z$ be an irreducible subvariety of $X$ which is a maximal log-canonical center of $(X,D)$. Let $A$ be any ample $\QQ$-line bundle. There is an effective $\QQ$-divisor (which we also denote by $A$) whose associated line bundle is $A$ such that we still have the pair $(X,D+A)$ log-canonical and $Z \subset X$ a maximal log-canonical center of $(X,D+A)$. Let $L$ be the $\QQ$-line bundle on $\xr$ given by $\OO(D+A)$ on $\xr$. We denote the $\QQ$-line bundle $\OO(K_X + D) \otimes \OO(A)$ on $X$ by $K_X + L$.  Let $D_1 = D+A$. 

\begin{theorem}[$L^2$ extension]\label{main}
 
 Let $Z \subset X$ be a maximal log-canonical center of a log-canonical pair $(X,D_1)$ where $D_1$ is an effective $\QQ$-divisor as above. Assume that $Z$ is not contained in $\xs$. Let $h$ be a Kawamata metric (Definition~\ref{kawamata}) of the log-canonical center $Z$ of the pair $(X,D_1)$. Then there exist
 
\begin{itemize}
\item
   a constant $C = C_{((X,D_1), Z)}$, 
\item   
   a hyperplane section $H \subset X$ and 
\item   
   a singular metric of the second kind $g = g_{((X,D_1), Z)}$ of $L$ which is bounded away from zero and whose domain is $ X \setminus H \subset \xr $ 
\end{itemize}
such that the following holds:  If given any  $\QQ$-line bundle $B$ on $X$ with $(K_X + L) + B$ being an integral line bundle, any singular hermitian metric $b$ of the first kind of $B$ on $X \setminus H$ and any holomorphic section $s \in \Gm (Z, (K_X+L)|_Z + B|_Z) $ satisfying 
 
\begin{align}\label{given}
 \int_Z |s|^2 \cdot h \cdot {b|_Z} < \infty  ,
\end{align}
	
\noindent \text{then there exists a holomorphic section} $ \st \in \Gm (X, (K_X + L) + B) $ \text{such that} $ \st|_Z = s $ \text{and}  

\begin{align*}
\int_X |\st|^2 \cdot g \cdot b \le C \int_Z |s|^2 \cdot h \cdot {b|_Z} .
\end{align*}

\noindent The constant $ C = C_{( (X,D_1), Z )} $ and the singular metric $g = g_{((X,D_1), Z)}$ of $L$ are independent of $(B,b)$ and the section $s$.     \qa \qa \qa \qa \qa \qa \qa \qa \qa \qa  (end of the statement) 

\end{theorem}

\qa

 The condition on $g$ to be bounded away from zero is precisely what we need in the proof of this theorem (in Step 7) and in its application (for example, in (\ref{extend sigma})).

 The proof of Theorem~\ref{main} is in the next section. To construct the function $\ld$ mentioned before the statement, we apply Siu's theorem on global generation of multiplier ideal sheaves to the sheaf \eqref{siu global}. We take the $q$-th roots $s_1, \cdots, s_k$ of $k$ of the generating global sections and take \eqref{def lambda} in Step 2. The use of an arbitrary ample $\QQ$-line bundle $A$ in the statement is completely limited to this step. We note that, for any positive integer $a \ge 1$, we can use $\frac{1}{a} A$ the same way : for the line bundle $ K_X + L_a = \OO( K_X + D + \frac{1}{a} A)$, we take $aq$-th roots of sections of

\begin{equation}\label{siu g} 
   aq{L_a} \otimes \JJ(aqD) = \OO( K_X + p A_0 + aqD) \otimes \JJ(aqD) 
\end{equation}
  
\noindent instead of \eqref{siu global}. This gives a sequence of functions $\{ \lambda_a \}$ ($ a \ge 1$) except the special case of the lc center $Z$ being a Cartier divisor in $X$. For a simple example, suppose that $Z$ is a smooth divisor and $ D = Z $. Then the multiplier ideal sheaf $\JJ(aqD)$ is equal to the line bundle $\OO (-aqD)$ and the sheaf in \eqref{siu g} is constantly $\OO(K_X + pA_0)$ for any $aq$. So there is no sequence whose limit to take: on the other hand, for a divisor case without $A$, we have the following example where $L^2$ extension cannot be obtained (since $L^2$ extension as in Theorem~\ref{main} implies pluriadjoint extension as in Theorem~\ref{pluriadjoint} as we will see in Section 5). 

\begin{example}\label{ex1}
 Let $Y$ be a smooth projective variety which is a fiber of the product  $ X := Y \times \PP^1 $.  Then no multiple $\OO(m(K_X+Y))$ has a nonzero holomorphic section while we can take $Y$ to be one with many sections of $\OO_{Y} (mK_Y)$ for $m \ge 1$. So we do not have surjectivity of the restriction map $ \Gm (X, \OO(m(K_X+Y))) \to \Gm( Y, \OO(mK_Y))$ for any $m \ge 1$. 
\end{example}
   
\noindent In typical application of $L^2$ extension in algebraic geometry, the interest is in the existence of a section of $K_X + L$. The special case of $L$ being equivalent to $Z + D'$ where $Z$ is a Cartier divisor and $D' \ge 0$, is either essentially equivalent to the existence of a section or reduces the existence of a section to a smaller line bundle. Such a case will be excluded in a modified setting of lc centers.

\subsection{Proof of the main theorem }

 The proof of Theorem~\ref{main} is divided into the following steps. 
\\

{\small

Step 0. Choice of a hyperplane section $ H \subset X $

Step 1. A tubular neighborhood of $Z$ given by the union of open sets $W_{\el}$ or $V_\el$ 
 
Step 2. Construction of the function $\lambda : \om_t \to \RR $

Step 3. Setup of the $\db$ equation

Step 4. Introducing two factors I* and II*

Step 5. Inequality II $\ge$ II*

Step 6. Inequality I $\ge$ I*

Step 7. From each $\om_t$ to $X \setminus H$, to $X$ 
}

\qa

\noindent In Step 0, we first choose multi-valued holomorphic sections $s_1, \cdots, s_k$ of $L$ cutting out $\JJ (D)$ on a Zariski open subset of $\xr$, which will be used in Step 2, as explained in the previous section. Then we choose a hyperplane section $H \subset X$ satisfying appropriate conditions and most of our steps in this proof will be on the complement $X \setminus H$ to obtain the wanted extension on $X \setminus H$ in Step 7. At the end of Step 7, we apply our version of the Riemann extension theorem, Proposition~\ref{riemann extend}, to extend the section on $X \setminus H$ across $H$, to $X$. 
 
\noindent More precisely, the $\db$ equation is defined and solved ( Steps 2,3,4,5,6 ) on each $\om_t$, a member of an increasing exhaustion sequence of relatively compact Stein open subsets so that $ \cup_{t \ge 1} \om_t  = X \setminus H $ (as in the setup before Proposition~\ref{montel}).

\subsubsection{Setup of the $\db$ equation}

\textit{ \newline \textbf{Step 0}. Choice of a hyperplane section $ H \subset X $ }
\\

\noindent First we fix a very ample integral line bundle $A_0$ on $X$. For the ample $\QQ$-line bundle $A$, we can write $ \displaystyle  A = \frac{p}{q} A_0 + \frac{1}{q} K_X $ with some integers $ p \ge n+1 $ and $ q > 1$. Then by Siu's theorem on global generation of multiplier ideal sheaves (\cite{Siu98} Proposition 1, also \cite{L} (9.4.26)), the sheaf on $\xr$ ( with $qL$ an integral line bundle )

\begin{equation}\label{siu global} 
  qL \otimes \JJ(qD) = \OO( K_X + p A_0 + qD) \otimes \JJ(qD) 
\end{equation}

\noindent is generated by its global sections $\Gamma$ on $\xr$. We have the subadditivity property (\cite{L} (9.5.20)) $ \JJ (qD) \subseteq ( \JJ(D) )^q $. Then there is a proper (possibly reducible) subvariety $X_1 \subsetneq X$ given by the image of some exceptional divisors under the log resolution of $(X,D)$, such that $ \JJ(qD) = ( \JJ(D) )^q $ on the open complement $ X \setminus X_1 $. Moreover, we can choose $k$ multi-valued sections $s_1, \cdots, s_k$  (being the $q$-th roots of $k$ sections of $\Gamma$) such that they give the local equations of $\zr$ around each point of $\zr \backslash (X_1 \cup X_2)$ where $X_2 \subsetneq X$ is another proper (possibly reducible) subvariety of $X$.  Recall that the open subset $Z_0 \subset Z$ is the domain of the Kawamata metric $h$. 
 
 Let $ H \subset X $ be a hyperplane section in a projective embedding of $X \subset {\mathbb{P}}^N$ such that 
 
\begin{itemize}
\item
 $ Z \nsubseteq H$. 
\item 
 $  (  \xs  \cup \zs \cup (Z \setminus Z_0) \cup X_1 \cup X_2 ) \subset H $.
\item 
 $H$ contains the divisor $\divisor(s)$ (i.e. the zero set and the pole set) of a meromorphic section $s$ of $L$ on $X$ so that the line bundle is trivialized on $X \setminus H$. We choose $s$ such that $Z \nsubseteq \divisor(s)$. 
\end{itemize} 
  
\noindent In addition, take another divisor $ H_B \subset X $, a hyperplane section in a projective embedding of $X \subset {\mathbb{P}}^N$ such that 
 
\begin{itemize} 

\item
 $ Z \nsubseteq H_B$. 
\item 
 $H_B$ contains the divisor $\divisor(s)$ (i.e. the zero set and the pole set) of a meromorphic section $s$ of $B$ on $X$ so that the line bundle is trivialized on $X \setminus H_B$. We choose $s$ such that $Z \nsubseteq \divisor(s)$. 
 
\end{itemize}

\noindent We fix an increasing exhaustion sequence of relatively compact Stein open subsets $ \{ \om_t \}_{t \ge 1} $ of the affine variety $ X \setminus (H \cup H_B)$ as in Section \ref{Stein}. 
 
 Now let $g_1$ be the singular metric of the first kind on $\xr$ associated to the effective $\QQ$-divisor $D_1$.  Since the line bundle $L$ is trivialized on $X \setminus (H \cup H_B)$, $g_1$ is given by a single function $e^{-\varphi}$ where $\varphi$ is a psh function on $X \setminus H$. On each $\om_t$, one can use the holomorphic tangent vector fields to regularize the psh function $\varphi$ by \cite{Siu98}. We fix one such sequence $g_\nu ( = g_{1,t,\nu})$ of regularizing smooth hermitian metrics of $g_1$ on $\om_t$ such that the weight function of $g_\nu$ converges to that of $g_1$ as the index $\nu \in \ZZ_{>0}$ goes to $\infty$. Similarly to $(L,g_1)$, we regularize the singular metric $(B,b)$ on each $\om_t$ and denote the sequence of regularized metrics by $b_\nu   (\nu = 1,2,3, \cdots)$ converging to $b$ as $ \nu \to \infty$.

\textit{ \newline \textbf{Step 1}. A tubular neighborhood of $Z$ given by the union of open sets $W_{\el}$ or $V_\el$ } 
\\

\noindent To setup our $\db$ equation, we need to choose and fix a finite collection of open subsets of $X \setminus H$ whose union contains $Z \setminus H$. We will have two different kinds ($W$'s and $V$'s) of such collection of open subsets, both of which can be regarded as giving a tubular neighborhood of the subvariety $Z \setminus H$. 

 First, we take and fix a finite collection of open sets $W_1, \cdots, W_{\mu_0} $ of $X \setminus H$ such that $ W_\el \cap Z \ne \emptyset $ for each $\el$ and $ (Z \setminus H) \subset W_1 \cup \cdots \cup W_{\mu_0} $. On each $W_\el$, we take a local analytic coordinate system $ ( z_1^{(\el)}, \cdots, z_n^{(\el)} ) $ where the solution set of  $ \{ z_1^{(\el)} = 0 , \cdots, z_k^{(\el)} = 0  \} $ gives $ Z \cap W_\el $ and moreover we can assume that  
 
 $$ \qa W_\el = \{ ( z_1^{(\el)}, \cdots,  z_n^{(\el)} ) \; \vert \qa   \zsum < {\ep_0}^{k+1} ,  \sum_{j=k+1}^n |z_j^{(\el)}|^2  < 1       \} $$
 
\noindent  for $ \exists \ep_0 > 0 $. For each choice of such an analytic coordinate system, we let
  
  $$ W_\el (\ep) :=   \{ ( z_1^{(\el)}, \cdots,  z_n^{(\el)} ) \; \vert \qa  \zsum < {\ep}^{k+1} , \sum_{j=k+1}^n |z_j^{(\el)}|^2  < 1  \} $$
  
\noindent  for $ \ep < \ep_0$. Note that $W_\el (\ep)$ is a Stein manifold since it is the product of two Stein manifolds. 

 Second, for $ \ep < \ep_0$, we take another finite collection of open subsets $V_1 (\ep), \cdots, V_{\mu} (\ep)$ such that each $V_\el (\ep)$ is contained in some $W_{\el'} (\ep)$ and moreover, $V_\el (\ep)$ is the product of the set $ \{ \zsuml < {\ep}^{k+1} \} $ and an open subset of $ \{ \sum_{j=k+1}^n |z_j^{(\el')}|^2  < 1 \} $.   Unlike $W_{\el'}$'s, we do not need $V_\el$ to be Stein but we require the overlaps between different $V_\el$'s to be sufficiently small. More precisely, let $\omega$ be the volume of the set of points in $V_1 (\ep) \cup \cdots \cup V_{\mu} (\ep)$ belonging to more than one $V_\el (\ep)$. Then $\omega$ is a function of $\ep$, and $\omega$ is sufficiently small when we take the limit $\ep \to 0$ later. We use the fact that $\omega$ is sufficiently small at one point, when we use the Twisted Basic Estimate after Lemma~\ref{local}. We note that we can obtain these $V_\el (\ep)$'s by replacing each $W_{\el'}$ by the union of small enough open sets $V_{\el}$ of the above product type, whose union may leave some part of $W_{\el'}$ uncovered. 
 We will often use the same $\el$ to denote the index both for $W$'s and for $V$'s, which will not cause confusion. \textit{The index $\el$ for $V_{\el}$ should also be interpreted as equal to the index $\el'$ for one $W_{\el'}$ containing $V_{\el}$, thus allowing $\el'$ to be denoted by $\el$. } 
 
 To define the right hand side of our $\db$ equation in Step 3, we need to take unconditioned local extension of the given section $s \in \Gm (Z,  (K_X+L)|_Z + B|_Z)$ from each $Z \cap W_\el$ to $W_\el$. So we fix the following data, the first for $W$'s and the second for $V$'s:

\begin{itemize}
\item 
 First, on each $W_{\el}$,  a local frame (i.e. a local nonvanishing section) $\theta^L_\el $ of $L$, a local frame $\theta^B_\el $ of $B$ for each $\el \in \{ 1, \cdots, \mu_0 \} $.  Also the local frame $\theta^K_\el $ of $K_X$ determined by an orthonormal coframe $ \mg_1, \cdots, \mg_n $ in $W_{\el}$, as in Section \ref{hilbert}. Denote the product $\theta^L_\el \theta^B_\el$ by  $\theta_\el$. We have the local frame $\theta^K_\el \theta_\el$ of the line bundle $K_X + L + B $ on $W_\el$. 
\item
 Second, a $C^\infty$ partition of unity $\vartheta_1, \cdots, \vartheta_\mu $ subordinate to the covering $ \{ V_\el \} $ such that $ \sum \vartheta_\el = 1 $ in a neighborhood of $Z \setminus H$.  
\end{itemize} 
 
\noindent  If the given section $s$ is represented by a holomorphic function $a \in \OO_{Z \cap W_\el}$ up to the above local frames in $W_{\el}$, that is, if $ s|_{V_\el} = a \cdot {\theta^K_\el \theta_\el}|_Z $, then we set the local extension on $W_\el$ to be

 $$ {\st}_\el := \widetilde{a}_\el    \cdot \theta^K_\el \theta_\el $$
 
\noindent where $\widetilde{a}_\el  \in \OO_{W_\el} $ is a holomorphic extension of $a$ (that is, ${\widetilde{a}_\el}|_Z = a$) in $W_\el$ which simply exists since $W_\el$ is Stein.  We do not need any particular condition on this local extension $ {\st}_\el $. Now using the above partition of unity, we define a $(L+B)$-valued $(n,0)$ form on $V_\el$ (note our convention of using the index $\el$ between $V$'s and $W$'s as in the above ) by 

$$ \sigma_\el (\ep) :=  \chi \left( \frac{\sum_{i=1}^k \vert z_i^{(\el)}  \vert^2}{\ep^{k+1}} \right) \cdot \vartheta_\el \cdot {\st}_\el $$

\noindent where $\chi$ is a fixed cut-off function of one real variable as in \cite{Siu02}, p.246. That is, the support of $\chi$ is in $[ 0,1 ]$, $ \chi \equiv 1 $ on $[0, \frac{\delta}{2}]$ and $ \abs{ \chi'(x) }  \le  1 + \delta $ for $ x \in [0,1]$ where $0 < \delta \ll 1$ is a constant. We do not need to let $\delta \to 0$.

\textit{ \newline \textbf{Step 2}. Construction of the function $\lambda = \ld(t, \nu, \ep) : \om_t \to \RR_{\ge 1} $ }
\\
  
\noindent Since $s_1, \cdots, s_k$ from Step 0 generate $\JJ(D)$ on $ X \setminus H $, there exists a constant $\tau_0 > 0$ such that $ \sum_{j=1}^k \abs{s_j}_{g_\nu }^2 \le \tau_0 $ for all $\nu \ge 1$.  We take the following family of $\RR$-valued functions 
 
\begin{equation}\label{def lambda} 
  \ld = \ld(t, \nu, \ep, \tau)   = \tau - \log ( \sum_{j=1}^k \abs{s_j}_{g_\nu }^2  + \hat{\ep}^2 )  
\end{equation}

\noindent where $ \hat{\ep} = \ep \cdot g_\nu $ (note that the metric $g_\nu$ is given as a single function, say $e^{-\varphi_\nu}$ on $\om_t$),     $ 0 < \ep < \ep_0 , m \in \ZZ_{>0} \;\text{and}\; \tau \ge 1 + \log ( \tau_0 + \hat{\ep_0}^2)$. Then for all $(t, \nu, \ep, \tau)$, the function satisfies $\ld(t, \nu, \ep, \tau) \ge 1 $ on $\om_t$ and also as real smooth $(1,1)$ forms

\begin{equation}\label{condition1}
\sqrt{-1} \Theta_{g_\nu} (L) + \iddb(-\ld(t, \nu, \ep, \tau))  \ge 0 
\end{equation}
 
\noindent on $\om_t$ and

\begin{equation}\label{condition2}
\sqrt{-1} \Theta_{g_\nu} (L) + \iddb(-\ld(t, \nu, \ep, \tau))  \ge  \iddb \log ( \sum_{i=1}^k \abs{z_i^{(\el)}}^2 + \ep^2 )
\end{equation} 

\noindent on $\om_t \cap V_\el$ for each $\el$. 
\\

\textit{\newline \textbf{Step 3}. Setup of the $\db$ equation }
\\

\noindent We formulate our main $\db$ equation in terms of Hilbert spaces $\HH_q := \LL^2_{(n,q)} ( \om_t, L+B, g_\nu b_\nu ) $ for $ q = 0,1,2 $. The $\db$ equation and its solution is in terms of the indices $(t, \nu, \ep)$, fixing one value of $\tau$ for which we do not take a limit. Later we take the limit involving the solution as $ \ep \to 0,  \nu \to \infty$ and $ t \to \infty$.

 Following \cite{MV}, \cite{Var}, we use the functions $\eta = \ld + r ( \ld) $ and $ \gamma = \displaystyle \frac{(1+r'(\ld))^2}{-r''(\ld)} $ for each case of $\ld = \ld(t, \nu, \ep, \tau)$ to define the modified $\db$ operators $ T := \db ( (\sqrt{ \eta + \gamma }) \; \cdot ) $ and $ S := (\sqrt{\eta}) \db ( \cdot) $ as in the discussion before Proposition~\ref{BK0}. Note the domains and ranges: $ T: \HH_0 \to \HH_1$ and $ S: \HH_1 \to \HH_2$. Now our $\db$ equation is

\begin{equation}\label{mainequation}
 T v = \al_\ep := \db (\sum^{\mu}_{ \el = 1} \sigma_\el (\ep) ).
\end{equation}

\noindent where the $(L+B)$-valued $(n,0)$ form $\sigma_\el (\ep)$ is as defined at the end of Step 1. 
\\

\subsubsection{Two main inequalities and the extension}

\textit{\newline \textbf{Step 4}. Introducing two factors I* and II* }
\\

\noindent It is standard by (\ref{FA}) that solving \eqref{mainequation} (in the sense of (\ref{FA})) is equivalent to showing that there exists a constant $C_2$ satisfying the inequality

\begin{equation}\label{apriori}
 \abs{\inner{ u, \al_\ep }}^2  \le  ( C_2  \int_Z \abs{s}^2 \cdot h \cdot b|_Z ) \cdot  ( \norm{T^* u}^2 + \Vert S u \Vert^2 ) =: \text{I} \cdot \text{II}
\end{equation}

\noindent for all $ u \in \Dom(T^*) \cap \Dom(S) \subset \HH_1$. We will do this for all sufficiently small $ \ep > 0 $. We denote the first factor of (\ref{apriori}) by I and the second by II.

 First, we have the following inequalities for the left hand side of \eqref{apriori} by the fact that $ \sigma_\el (\ep)$ is supported on $ V_\el (\ep)$ and the Cauchy-Schwarz inequality. 

\begin{align}
 \abs{\li u, \al_\ep \ri}^2  &= \abs{ \int_{\om_t}  \langle u, \sum^{\mu}_{\el=1} \db \sigma_\el (\ep) \rangle_{g_\nu b_\nu} dV   }^2  \notag   \\
 &\le \abs{ \int_{V_1(\ep) \cap \om_t} \lan u, \db \sigma_1 (\ep) \ran_{g_\nu b_\nu} dV + \cdots  + \int_{V_\mu(\ep) \cap \om_t} \lan u, \db \sigma_\mu (\ep) \ran_{g_\nu b_\nu} dV }^2     \notag      \\
 &\le  \mu \cdot \sum^{\mu}_{\el=1} \abs{ \int_{V_\el (\ep) \cap \om_t} \lan u, \db \sigma_\el (\ep) \ran_{g_\nu b_\nu} dV }^2   =:    \mu \cdot \sum^{\mu}_{\el=1} S_\el     \label{step4}
\end{align}

\noindent In order to take a local expression in $V_\el$ of each summand $S_\el$ of the last line, we fix an orthonormal basis of $(n,1)$ forms  $ \omega_I \wedge \mg_1 , \cdots, \omega_I \wedge \mg_n $ where $ \omega_I $ is the $(n,0)$ form $ \omega_1 \wedge \cdots \wedge \omega_n $. We then write $  u = \sum^n_{i=1} u_i \theta_\el \otimes \omega_I \wedge \mg_i  $ in $ V_\el$ where $\theta_\el$ is a local frame of $L+B$ we fixed before. Let $e^{-\varphi} = g_\nu b_\nu (\theta_\el, \theta_\el)$. Now we consider 
 
\begin{equation}\label{sigmal}
 \db \sigma_\el (\ep) = \frac{1}{ \ep^{k+1} } \cdot {{\chi}^{\prime}}  \cdot  \db (\sum_{i=1}^k \vert z_i^{(\el)}  \vert^2) \cdot \vartheta_\el \cdot {\st}_\el + \chi \left( \frac{\sum_{i=1}^k \abs{  z_i^{(\el)}   }^2}{\ep^{k+1}} \right) \cdot \db ( \vartheta_\el \cdot {\st}_\el ). 
\end{equation}

\noindent Determine the component functions $ \zeta_i$'s by writing $  \db (\sum_{i=1}^k \vert z_i^{(\el)}  \vert^2) = \sum_{i=1}^k z_i d \zb_i = \sum_{i=1}^k \zeta_i \mg_i.$   Since $ \int_{V_\el (\ep) \cap \om_t} \abs{ \inp{u,    \chi \left( \frac{\sum_{i=1}^k \abs{  z_i^{(\el)}   }^2}{\ep^{k+1}} \right) \cdot \db ( \vartheta_\el \cdot {\st}_\el )  }_g } dV $ goes to zero as $ \ep \rrow 0 $, it suffices to consider only the first term of the right hand side of \eqref{sigmal} to be taken inner product with $u$ for sufficiently small $\ep >0$. So we have the following, for a constant $0.9 < C_7 < 1$ which is independent of $u$ and $(t,\nu,\ep)$ (also defining ${\st'}_\el$ by $\st_\el = {\st'}_\el \omega_I $):

\begin{align} 
 \qa  & C_7 \cdot S_\el \le \left( \int_{V_\el (\ep) \cap \om_t }  \sum_{i=1}^k   \abs{ u_i \zt_i  \frac{\chi \prime \cdot \vartheta_\el \cdot {\st'}_\el}{ \ep^{k+1} } } e^{-\varphi} dV \right)^2  \notag  \\
 \le &  \left( \int_{V_\el (\ep) \cap \om_t} |{\st'}_\el|^2 ( \sum_{i=1}^k |\zt_i|^2 )  \frac{| \chi \prime |^2}{\ep^{2k+2}}   |\vartheta_\el|^2         \frac{K^2}{\ep^2}   e^{-\varphi} dV   \right)       \left( \int_{V_\el (\ep) \cap \om_t} ( \sum_{i=1}^k \ve u_i \ve^2 )  \frac{\ep^2}{K^2} e^{-\varphi} dV  \right)       \label{CS}  \\
 \le & \left( \frac{ C_1 }{\ep^{2k}} \int_{V_\el (\ep) \cap \om_t} |{\st'}_\el|^2 ( \sum_{i=1}^k |\zt_i|^2 ) e^{-\varphi} dV    \right)  \left( \int_{V_\el (\ep) \cap \om_t} ( \sum_{i=1}^k \ve u_i \ve^2 )  \frac{\ep^2}{K^2} e^{-\varphi}  dV  \right)  =: \frac{1}{\mu} \text{I}^*_{\el} \cdot \text{II}^*_{\el}  \label{factors}  
\end{align}

\noindent for a positive constant $C_1$, using Cauchy-Schwarz and introducing the factor $ \displaystyle \frac{K^2}{\ep^2} $ where $ K := \zsum + \ep^2 $. We call $\mu$ times the first factor of (\ref{factors}) as $\text{I}^*_{\el}$ and the second factor as $\text{II}^*_{\el}$. We will show the inequalities of the types I $\ge \text{I}^*_{\el}$ and  \mbox{II $\ge$ II*  $:= \sum_\el \text{II}^*_{\el}$} (up to some constants multiplied) relating \eqref{factors} and \eqref{apriori}. 
\\

\textit{\newline \textbf{Step 5}. Inequality II $\ge$ II*  }
\\

\noindent The actual inequality we will have is not II $\ge$ II*, but II $\ge C_6 \cdot \text{II*}$ for a constant $C_6$ as we will see below. We start with the following lemma, which is local calculation in $V_\el$.

\begin{lemma}\label{local}

 Let $\kap(\ep)$ be the function $ \log ( \zsum + \ep^2 ) = \log K $. Then we have the inequality
 
 $$ ( \iddb (\kap(\ep))(u,u)_{g_\nu b_\nu}  \ge  \frac{ \ep^2}{K^2 } \cdot ( \ve u_1 \ve^2 + \cdots + \ve u_k \ve^2 ) e^{-\varphi} .$$

\end{lemma}

\begin{proof}

 For simplicity in notation, we suppress the notation of the metric $ g_\nu b_\nu = e^{-\varphi} $ in the following.  
 Using the second derivatives (for $ 1 \le i, j \le n , i \neq j $)
 
$$ \frac{ \de^2 \kap(\ep) }{ \de \mg_i \de \og_i } = \frac{ \zsumoe - \ztb_i \cdot \zt_i }{ K^2} \qa \text{and} \qa \frac{ \de^2 \kap(\ep) }{ \de \mg_j \de \og_i } = \frac{ - \ztb_i \cdot \zt_j }{ K^2} \qa ,  $$

\noindent we have the left hand side equal to

\begin{align*}
&= \sum_{j=1}^k  \frac{ \zsumoe - \ve \zt_j \ve^2 }{K^2} \cdot \ve u_j \ve^2   -  \frac{1}{K^2} \cdot \sum_{1\le i < j \le k} ( \zt_i \ztb_j u_i \ub_j + \ztb_i \zt_j \ub_i u_j )\\
&= \frac{1}{K} \cdot ( \ve u_1 \ve^2 + \cdots + \ve u_k \ve^2) -  \frac{1}{K^2} \cdot \ve \sum_{j=1}^k u_j \zt_j \ve^2 \\
&= \frac{1}{K^2} \cdot \Big( ( \ve u_1 \ve^2 + \cdots + \ve u_k \ve^2) \cdot \ep^2 + (  \ve u_1 \ve^2 + \cdots + \ve u_k \ve^2)( \ve \zt_1 \ve^2 + \cdots + \ve \zt_k \ve^2 ) - \ve \sum_{j=1}^k u_j \zt_j \ve^2  \Big) \\
&\ge   \frac{1}{K^2} \cdot \Big( ( \ve u_1 \ve^2 + \cdots + \ve u_k \ve^2) \cdot \ep^2 \Big), 
\end{align*}

\noindent where the inequality holds by Cauchy-Schwarz. Note that the inequality degenerates to an equality when $Z$ is of codimension $1$.  

\end{proof}

\noi   Next, we use Proposition~\ref{BK0} (Twisted Basic Estimate of \cite{MV}) for each regularized metric $g_\nu b_\nu$ of $L+B$ and $ \ep >  0$ (so that $\ld$ and $\eta$ are $C^2$) to get: 
 
\begin{align*}
\norm{T^* u}^2 + \Vert S u \Vert^2 &\ge \int_{\om_t} ( \eta \sqrt{-1} \Theta_{g_\nu b_\nu}(L + B) - \sqrt{-1} \partial \overline{\partial} \eta - \frac{1}{\gamma} \sqrt{-1} \partial \eta \wedge \overline{\partial} \eta ) (u,u)_{g_\nu b_\nu} dV \\
& = \int_{\om_t} ( \eta \sqrt{-1} \Theta_{b_\nu}(B)  +  \eta \sqrt{-1} \Theta_{g_\nu} (L) + ( 1 + r'(\lambda)) ( -\iddb \lambda) ) (u,u)_{g_\nu b_\nu} dV   \\
& \ge C_6 \cdot \sum_{\el=1}^\mu \int_{V_\el(\ep) \cap \om_t} ( \iddb   \log ( \zsum + \ep^2 ) ) (u,u)_{g_\nu b_\nu} dV \\
& \ge C_6 \cdot \sum_{\el=1}^\mu \int_{V_\el(\ep) \cap \om_t} (\ve u_1 \ve^2 + \cdots + \ve u_k \ve^2 ) \cdot \frac{\ep^2}{K^2} e^{-\varphi} dV   \; = C_6 \sum_{\el=1}^\mu \text{II}^*_{\el} = C_6 \cdot \text{II*}
\end{align*}

\noindent which gives II $\ge C_6$ II*, where $0.9< C_6 < 1$ is a constant which appears from the fact that there is a small overlap between $V_\el (\ep)$'s for sufficiently small $\ep >0$, as mentioned in Step 1. $C_6$ is independent of $u$ and $(t,\nu,\ep)$.  For the second inequality, we used \eqref{eta}, \eqref{condition1}, \eqref{condition2} and $ \sqrt{-1} \Theta_{b_\nu}(B) \ge 0, \;  \sqrt{-1} \Theta_{g_\nu} (L)  \ge 0$. For the third inequality, we used Lemma~\ref{local}. 
\\

\textit{\newline \textbf{Step 6}. Inequality I $\ge$ I*  } 
\\ 
 
\noindent The actual inequality we will have is not I $\ge \text{I}^*_\el$, but I $\ge \frac{1}{C_7 C_6} \cdot \text{I}^*_\el$ for $C_7$ from Step 4. The inequality is in the sense that we can choose a constant $C_2$. First, for $\frac{1}{\mu} \text{I}^*_{\el}$ of \eqref{factors}, we have the inequality  
 
\begin{equation}\label{Istar} 
 \frac{ C_1 }{\ep^{2k}}  \int_{V_\el(\ep) \cap \om_t} |\st_\el|^2 \cdot \hat{g}_\nu b_\nu   \le \frac{ C_1 }{\ep^{2k}}  \int_{W_\el(\ep)} |\st_\el|^2 \cdot \hat{g}_1 b 
\end{equation}

\noindent where $ \hat{g}_\nu := g_\nu \cdot ( |\zt_1|^2 + \cdots + |\zt_k|^2 ) $ as a metric of $L$ over $W_\el$. Recall that the sequence of smooth hermitian metrics $g_\nu$ gives regularization of the singular hermitian metric $g_1$ as in Step 0.  Since regularization of a psh function converges from the above, we have $ |\st_\el|^2 \cdot \hat{g}_\nu b_\nu \le  |\st_\el|^2 \cdot \hat{g_1} b $. We also used $ V_\el(\ep) \cap \om_t \subset W_\el(\ep)$.

 The key in this step is to show that (the right hand side of) \eqref{Istar} is finite.  By Proposition~\ref{change}, we first have $ \int_{W_\el(\ep)} |\st_\el|^2 \cdot \hat{g_1} b  = \int_{ f^{-1}(W_\el(\ep))}  \abs{ f^* \st_\el }^2 \cdot ({\hat{g_1}} b)' $ where $f$ is a log-resolution of $(X,D_1)$ as in \eqref{rel}. Then we will apply Theorem~\ref{kltklt} to the lc center $Z$ of the pair $(X, D_1)$ putting $W_\el(\ep)$ in the place of $V$, an open Stein subset of $X$.  Following the notation in Section 3 and \eqref{rel}, we write 
 
\begin{equation}\label{rel1} 
 K_{X'} = f^* (K_X + D_1) - E -(D_1)' - \Delta
\end{equation}

\noindent where $(D_1)'$ is the birational transform of $D_1$ under $f$, $\Delta$ a combination of exceptional divisors and $E$ the exceptional divisor over $Z$.

   Note that the section $\st_\el$ restricts to $s \in H^0 (W_\el \cap Z, (K_X+L)|_Z ) $ on $Z$ which satisfies $  \int_{ W_\el \cap Z} \abs{ s }^2 \cdot h \cdot b|_Z  < \infty $. Thus Theorem~\ref{kltklt} gives

  $$ \int_{ f^{-1} (W_\el(\ep)) }  \abs{ f^* \st_\el }^2 \cdot \eta_{( (D_1)' + \Delta)} \cdot \gamma_{\OO (E)} \cdot f^* b   < \infty $$ 
  
\noindent where $\gamma_{\OO (E)}$ is any smooth metric of the line bundle $\OO(E)$. It follows from this and \eqref{rel1} that
 
 $$ \int_{ f^{-1}(W_\el(\ep))}  \abs{ f^* \st_\el }^2 \cdot f^*(g_1) f^* ( |\zt_1|^2 + \cdots + |\zt_k|^2 ) \cdot f^* b   < \infty    $$

\noindent where $f^*(g_1)$ is the singular metric associated to the divisor $f^*{D_1}$ and $ f^* ( |\zt_1|^2 + \cdots + |\zt_k|^2 ) $ gives the multiplication of a local equation of $E$. Thus \eqref{Istar} is finite. 

 Once the finiteness is shown, we only need to observe the following: Up to local frames, the sections $\st_\el$ and $s$ are given by holomorphic functions $a$ and $a|_Z$, respectively ($a \in \OO_{W_\el (\ep)}$). Then  $ \int_{V_\el(\ep)} |\st_\el|^2 \cdot \hat{g}_1 b  < \infty $  is integrating $ \abs{a}$ with respect to a $2n$ dimensional measure while $ \int_{ Z \cap V_\el(\ep)} \abs{s}^2 \cdot h \cdot b|_Z $ is integrating $ \abs{a}|_Z$ with respect to a $2(n-k)$ dimensional measure. Since the latter measure is not zero (that is, zero times the measure associated to a local euclidean volume form) in any open subset, there exists a constant $C'_\el$ such that 
 
 $$ \frac{ 1 }{\ep^{2k}}\int_{W_\el(\ep)} |\st_\el|^2 \cdot \hat{g}_1 b    \le   C'_\el   \int_{ Z \cap W_\el(\ep)} \abs{s}^2 \cdot h \cdot b|_Z    \le C'_\el   \int_{ Z } \abs{s}^2 \cdot h \cdot b|_Z $$
 
\noindent for $ \ep \ll 1 $. Taking $ \displaystyle C_2 = \frac{\mu}{C_7 C_6} C_1 \cdot (\max_{1 \le \el \le \mu} C'_\el) $, we have the inequality \mbox{$ \displaystyle \text{I} \ge \frac{\mu}{C_7 C_6} \frac{1}{\mu} \text{I}^*_{\el} = \frac{1}{C_7 C_6} \text{I}^*_{\el}  $}.
\\

\textit{\newline \textbf{Step 7}. From each $\om_t$ to $X \setminus H$, to $X$ }
\\

\noindent Now the inequalities I $\ge \frac{1}{C_7 C_6} \text{I}^*_{\el} $ and II $\ge C_6$ II* give \eqref{apriori}:

\begin{align*}
 \text{I} \cdot \text{II} \ge \frac{1}{C_7 C_6} \cdot  C_6 \sum_\el \text{I}^*_\el  \text{II}^*_\el   \ge \mu \sum_\el S_\el \ge \abs{\inner{ u, \al_\ep }}^2          
\end{align*}

\noindent where we used \eqref{factors} for the second inequality and \eqref{step4} for the third inequality. By Proposition~\ref{FA}, this solves the $\db$ equation \eqref{mainequation}, together with the estimate of the solution $v_\ep$,   $ \norm{ v_\ep }^2 \le C_2  \int_Z \abs{s}^2 \cdot h \cdot b|_Z  $.  We recall that the solution $ v_\ep $ is actually indexed by $( t, \nu, \ep )$, not only by $ \ep $. The right hand side of the estimate is independent of the index $( t, \nu, \ep )$.  We rewrite \eqref{mainequation} as
 
 $$  \db ( -\sqrt{ \eta + \gamma} \cdot v_\ep + \sum^{\mu}_{ \el = 1} \sigma_\el (\ep) ) = \db ( -\sqrt{ \eta + \gamma} \cdot v_\ep + \sum^{\mu}_{ \el = 1} \chi \left( \frac{\sum \vert {z_i}^{(\el)}  \vert^2}{\ep^{k+1}} \right) \cdot \vartheta_\el \cdot {\st}_\el  )= 0    ,$$

\noindent and put $ \displaystyle  F_{(t,\nu,\ep)} :=  -\sqrt{ \eta + \gamma} \cdot v_\ep + \sum^{\mu}_{ \el = 1} \chi \left( \frac{\sum \vert {z_i}^{(\el)}  \vert^2}{\ep^{k+1}} \right) \cdot \vartheta_\el \cdot {\st}_\el $, which is a $(L+B)$-valued holomorphic $(n,0)$ form, hence a holomorphic section of $\Gm (\om_t, K_X+L+B)$ and satisfies $ F_{(t,\nu,\ep)} |_Z = s$.

\noindent Now we define a singular metric of the second kind $g$ on $ X \backslash H $ for $L$ by 
 
 $$ \displaystyle g := \lim_{\substack{  \ep \to 0, \; \nu \to \infty, \\ t \to \infty }}  \frac{g_1}{  \sqrt{ \eta + \gamma} } .$$
 
\begin{lemma} 
 The singular hermitian metrics of the second kind $g$ and $g \cdot b$ are bounded away from zero. 
\end{lemma}

\begin{proof}

 The statement for $g \cdot b$ follows from the one for $g$ since $b$ is of the first kind and a psh function is locally bounded above.  Since $L$ is trivialized on $X \setminus H$, the metric $g_1$ is given by a single function $e^{-\varphi}$. Writing $ g_1 \frac{1}{  \sqrt{ \eta + \gamma} } = \exp ( -\varphi - \frac{1}{2} \log (\eta + \gamma) ) $, it suffices to show that $ \varphi + \frac{1}{2} \log (\eta + \gamma) $ is bounded above on $X \setminus H$ taking the limit, or equivalently, (*) locally bounded above there, since the closure of $X \setminus H$ is compact. 
 
 First, consider (*) away from $Z \setminus H$, that is, in each open subset of $X \setminus H$, disjoint from $Z \setminus H$.  The function $\varphi$ is locally bounded above since it is psh.  On the other hand, we have $ \eta + \gamma \le  1 + \log 2 + \ld + 2 e^{\ld -1} \le 1+ \ld + e^{\ld}  $ from before \eqref{eta}, thus it only remains to show that $\ld = \ld (\ep, \nu, t)$ is locally bounded above taking the limit, away from $Z \setminus H$. This follows from the definition of $\ld$, \eqref{def lambda}.

 Next, consider (*) near $Z \setminus H$, say, in an open neighborhood $U$ of a point of $Z \setminus H$. The function $\ld$ becomes large enough and goes to $+\infty$ as one approaches $Z \setminus H$ and as $\ep \to 0$. Thus, we have $ \frac{1}{2} \log (\eta + \gamma) \le \frac{1}{2} \log ( 1+ \ld + e^{\ld} ) \le \frac{1}{2} \log ( 2 e^{\ld} ) \le \ld $ on $U$. So it remains to show that $ \varphi + \ld$ is bounded above on $U$ taking the limit. From the choice of $H$ in Step 0, we have (modulo a bounded (both above and below) function on $U$) $ \varphi = \log ( \sum_{j=1}^k \abs{s_j}^2) $ whereas $\ld =  \tau - \log ( \sum_{j=1}^k \abs{s_j}_{g_\nu }^2  + \hat{\ep}^2 ) $. This completes the proof of the lemma.

\end{proof}

\noindent Since the volume of the support of $ F_{(t,\nu,\ep)} - \sqrt{ \eta + \gamma} \cdot v_\ep $ goes to zero as $ \ep \to 0 $ and $ \int_{X \setminus H} \abs{ \sqrt{ \eta_\ep + \gamma_\ep } \cdot v_\ep }^2 \cdot \frac{g_1}{  \sqrt{ \eta_\ep + \gamma_\ep } } \cdot b = \norm{ v_\ep }^2   $, there exists a sequence of pairs $(\nu_t, \ep_t)$ for $ t = 1, 2, 3, \cdots $ such that the sequence of sections $ s_t = F_{(t,\nu_t,\ep_t)} $ satisfies 
 
 $$ \int_{X \setminus ( H \cup H_B)} \abs{s_t}^2 \cdot g \cdot b  \le  C \int_Z \abs{s}^2 \cdot h \cdot b|_Z $$ 
 
\noindent for a constant $C > 0$, independent of $t$. We apply (\ref{montel}) to this sequence to obtain a section $\st_0$ on $X \setminus ( H \cup H_B)$. Since $X$ is normal, we can then apply (\ref{riemann extend}) to extend $\st_0$ across the divisor $H \cup H_B$ to obtain the wanted section $\st$ with \eqref{extend1}:   $$ \quad \;\;  \int_X |\st|^2 \cdot g \cdot b \le C \int_Z |s|^2 \cdot h \cdot {b|_Z} .$$  

\noi Considering the sequence of sections $ s_t |_Z - s $ on $Z$, it is easy to see that $ \st|_Z - s = 0$. This completes the proof of Theorem~\ref{main}.

\begin{remark}\label{remark46}

 One of the key points in the proof was \eqref{CS} where we introduced two factors by Cauchy-Schwarz. Note that the particular choice of the two factors made it possible to use two different fundamental properties of a (maximal) log-canonical center. Our use of Cauchy-Schwarz is adaptation to general codimension of the one in \cite{Siu02} for which \cite{Siu02} comments (before (3.1)):  \textit{...replaces the strictly positive (curvature) in all directions by the strictly positive (curvature) just for the direction normal to the hypersurface from which the holomorphic section is extended}.     The reader may also find it helpful to compare our use to the use of Cauchy-Schwarz in, for example, \cite{D97} (3.1). 

\end{remark}

\section{Pluriadjoint extension} 

 Siu~(\cite{Siu02}, \cite{Siu98}) invented and used an ingenious inductive argument of applying $L^2$ extension in order to extend pluricanonical and pluriadjoint sections.  P\u aun~\cite{Pa05} found a simplified and strengthened version of the argument, which we call the \textit{tower argument} and apply to Theorem~\ref{main}.

\noi Let $Z \subset X$, a $\QQ$-line bundle $K_X + L$ on $X$ and the Kawamata metric $h$ be as in Theorem~\ref{main}.  Let $ m \ge 1 $ be an integer such that $m(K_X+L)$ is an integral line bundle. On the complement of a hyperplane section $X \setminus H \subset \xr $ given in Theorem~\ref{main},  we fix a smooth metric for each of the line bundles $(K_X,g_K)$, $(L,g_L)$ and $(A,g_A)$.  Let $g^{(km+p)}$ denote the product smooth metric of the line bundle $(km+p)(K_X+L)+A$ given by products of $g_K, g_L$ and  $g_A$. 
 
\qa
\\

\noindent Throughout this section, we fix a global holomorphic section $ \sigma \in H^0(Z,m(K_X+L)|_Z)$ such that its $m$-th root $ \sigma^{\frac{1}{m}} $ as a multi-valued section of $(K_X+L)|_Z$ satisfies

\begin{equation}\label{fnorm} 
 \int_Z \abs{ \sigma^{\frac{1}{m}} }^2 \cdot h  <  \infty .
\end{equation}

\noi Let $m_0 \ge 1$ be the smallest integer such that $m_0 (K_X + L)$ is an integral line bundle. We (can always) choose an ample integral line bundle $A$ which is sufficiently ample such that the following hold:
 
\noindent For each $p = 0, 1, \cdots, m-1 $, there exist multi-valued sections $ \st_j^{(p)} \;   (j = 1 , \cdots, N_p) $ of the $\QQ$-line bundle $ p(K_X + L) +A $  such that

(A1) Each $ \st_j^{(p)} $ is divided by $ \sigma^{\frac{p}{m}} $     , and 

(A2) The $m_0$-th powers of $\displaystyle \frac{\st_j^{(p)}}{\sigma^{\frac{p}{m}}}$'s generate the line bundle $ m_0 ( p(K_X + L) + A -p(K_X + L)) = m_0 A $.
\\

\noindent It would be helpful for the reader also to interpret these properties for multi-valued sections in terms of their associated $\QQ$-divisors. In order to extend $\sigma$ to $X$, we need to be able to continue this sequence of sections $\st_j^{(p)}$ beyond $ 0 \le p \le m-1$ as follows.

\begin{proposition}\label{extend sigma}

 If $ \sigma \in H^0 ( Z, m(K_X + L)|_Z) $ satisfies the condition (*) below (in addition to \eqref{fnorm}, (A1) and (A2)), then $\sigma$ lies in the image of the natural restriction map 
 
 $$ H^0 ( X, m(K_X+L) ) \to H^0 (Z, m(K_X + L)|_Z) .$$

\noindent (*) There exist a constant $C_\diamondsuit$ and (for each $ k \ge 1 $ , $ p = 0, 1, \cdots, m-1$ and $ j = 1, \cdots, N_p $) multi-valued sections $ \st_j^{(km+p)} $ of the $\QQ$-line bundle $ (km+p)(K_X+L) + A $ such that the following hold (let $N_{-1} := N_{m-1}$): 
\\

(1) $$ \st_j^{(km+p)} |_Z  =  \sigma^{\otimes k} \otimes \st_j^{(p)}|_Z  .$$

(2) 

$$ \int_{X \backslash H} \frac{ \sum_{j=1}^{N_p} \abs{ \st_j^{(km+p)}   }^2_{g^{(km+p)}}          }{      \sum_{j=1}^{N_{p-1}}  \abs{ \st_j^{(km+p-1)}   }^2_{g^{(km+p-1)}    }}   dV   \le   C_\diamondsuit     .$$

\end{proposition}

\begin{proof}

\noindent  We would like to apply the $L^2$ extension Theorem~\ref{main} with $ B = (m-1)(K_X+L) $, for which we need the existence of a singular metric $(B,b)$ such that 

\begin{equation}\label{f2norm}
 \int_Z |\sigma|^2 \cdot b\cdot h < \infty.
\end{equation}

\noi We will construct $b$ using sections given in (*).  Consider the following function defined on $X \setminus H$:
 
 $$ f_{km+p} := \log ( \sum_{j=1}^{N_p}  \abs{\st_j^{(km+p)}   }^2_{g^{(km+p)}} )$$
 
\noindent for each $ k \ge 1$ and $ 0 \le p \le m-1 $.  It is well known from \cite{Siu02} and \cite{F06m} that the sequence of quasi-psh functions $\frac{1}{k} f_{km}  (k \ge 1)$ is locally uniformly bounded above. Since the sequence is a good family of quasi-psh functions (Definition~\ref{def qpsh}), its upper envelope is also a quasi-psh function on $X \setminus H$ by Proposition~\ref{qpsh}. We denote the upper envelope function by $f_\infty$. We note that 

 $$ \ii \Theta_{ (g_K g_L)^{m}}( m(K_X+L)) + \frac{1}{k} \ii \Theta_{g_A}(A) +  \iddb( \frac{1}{k} f_{km} ) \ge 0 .$$

\noindent Therefore, when we define a singular metric $h_\infty$ of $m(K_X+L)$ on $X \setminus H$ by $$ h_\infty := (g_K g_L)^{m} \cdot e^{-f_\infty} ,$$ we have
 
 $$ \ii \Theta_{h_\infty} ( m(K_X+L)) = \ii \Theta_{ (g_K g_L)^{m}}( m(K_X+L)) + \iddb f_\infty \ge 0 .$$

\noindent Take $ b = h_{\infty}^{\frac{m-1}{m}}   $ and we will show \eqref{f2norm}. We first have the upper bound of the following pointwise length with respect to the metric $  h_\infty |_Z    $: 
 
\begin{lemma}\label{last}  
 
 $ |\sigma|^2_{h_\infty |_Z} \le C_\clubsuit $ on $Z \setminus H$, for some $ C_\clubsuit > 0$.
 
\end{lemma}

\begin{proof}

 Note that 
 
\begin{align*} 
  ( \frac{1}{k} f_{km} )|_Z  &= \frac{1}{k} \log ( \sum_{j=1}^{N_0}  \abs{    (\st_j^{(km)})|_Z   }^2_{g^{(km)}} ) = \frac{1}{k} \log ( \sum_{j=1}^{N_0}  \abs{   \sigma^k \cdot  (\st_j^{(0)})|_Z   }^2_{g^{(km)}} )    \\
  &= \log (\abs{ \sigma }^2_{(g_K g_L)^m}) +   \frac{1}{k} \log ( \sum_{j=1}^{N_0}  \abs{  (\st_j^{(0)})|_Z   }^2_{g^{(0)}}  ) .
\end{align*}

\noindent From (A1) in the beginning, the sections $\st_j^{(0)}$ are base-point-free. So there is a lower bound $C_0 > 0$ with $ \sum_{j=1}^{N_0}  \abs{  (\st_j^{(0)})  }^2_{g^{(0)}}   \ge   C_0 > 0  $ for everywhere in $X$, in particular for everywhere in $Z$.  Thus, 
 
\begin{align*}
 \log (\abs{ \sigma }^2_{(g_K g_L)^m}) - (\frac{1}{k} f_{km} )|_Z   =  -\frac{1}{k} \log ( \sum_{j=1}^{N_0}  \abs{  (\st_j^{(0)})|_Z   }^2_{g^{(0)}} )   \le  -\frac{1}{k} \log ( C_0 )   \le  C_1 
\end{align*} 
 
\noindent where $C_1$ is a constant independent of $k$, defined by $ C_1 := 0 $ if $ C_0 \ge 1 $ and by $ C_1 := -\log(C_0) $ if $ C_0 < 1 $. The lemma is proved by taking the exponential of the last inequality. 
 
\end{proof}

\noindent Using this lemma, 
 
 $$ \int_Z |\sigma|^2 \cdot h_{\infty}^{\frac{m-1}{m}} \cdot h  
 = \int_{Z \setminus H} (|\sigma|^2_{h_\infty})^{\frac{m-1}{m}} |\sigma|^{\frac{2}{m}} \cdot h 
 \le   {C_\clubsuit}^{\frac{m-1}{m}} \int_{Z \setminus H} |{\sigma}^{\frac{1}{m}}|^2 \cdot h  < \infty  $$ 
 
\noindent where ${\sigma}^{\frac{1}{m}}$ gives a multi-valued holomorphic section of $ (K_X + L)|_Z $ whose adjoint norm with respect to $h$ is finite. We do not use the H\"older inequality here. Then by Theorem~\ref{main}, $\sigma$ is extended to $H^0(X, m(K_X+L))$. This completes the proof of Proposition~\ref{extend sigma}.

\end{proof}

\begin{theorem}\label{pluriadjoint}

 In the setting of Proposition~\ref{extend sigma}, suppose that $L$ is an integral line bundle. Then (*) of (\ref{extend sigma}) holds and therefore $\sigma$ in \eqref{fnorm} is extended to $X$.

\end{theorem}

\begin{proof}

 We first note that there exists a constant $C_1 > 0 $ such that

 $$ \max_{ 0 \le p \le m-1}  \sup_Z  \frac{    \sum_{j=1}^{N_p} \abs{ \st_j^{(p)}|_Z   }^2_{g^{(p)}}   } { \abs{\sigma^{\frac{1}{m}}}  (\sum_{j=1}^{N_{p-1}} \abs{\st_j^{(p-1)}|_Z   }^2_{g^{(p-1)}} ) } = C_1 $$
 
\noindent thanks to the properties (A1) and (A2) of $A$.

\noindent We will use induction on $ km + p $ to construct the required sections in (*).  Suppose $ k \ge 1 $ and assume that we have the required sections for $ km + p-1 $. The induction begins with $k=1$ and $p=0$.  We will apply the $L^2$ extension Theorem \ref{main}, to extend $\sigma^{\otimes k} \otimes \st_j^{(p)}|_Z $ by taking $ B = (km+p-1)(K_X + L) + A $ and $ b $ to be the singular metric given by the sections just constructed, i.e.  
 
 $$ b = \frac{  g^{(km+p-1)}  }{    \sum_{j=1}^{N_{p-1}} \abs{ \st_j^{(km+p-1)}   }^2_{g^{(km+p-1)}}              } .$$

\noindent Then the section on $Z$ to be extended satisfies the finiteness

\begin{align*} 
  \int_Z \abs{  \sigma^{\otimes k} \otimes \st_j^{(p)}|_Z  }^2 \cdot h \cdot b|_Z  
 &= \frac{   \abs{ \st_j^{(p)}|_Z   }^2_{g^{(p)}}   } { \abs{\sigma^{\frac{1}{m}}}  \sum_{j=1}^{N_{p-1}} \abs{\st_j^{(p-1)}|_Z   }^2_{g^{(p-1)}} } \int_{Z} \abs{ \sigma^{\frac{1}{m}} }^2 \cdot h   \\
 &\le \frac{    \sum_{j=1}^{N_p} \abs{ \st_j^{(p)}|_Z   }^2_{g^{(p)}}   } { \abs{\sigma^{\frac{1}{m}}}  \sum_{j=1}^{N_{p-1}} \abs{\st_j^{(p-1)}|_Z   }^2_{g^{(p-1)}} } \int_{Z}       \abs{ \sigma^{\frac{1}{m}} }^2 \cdot h   \le C_1  \int_{Z} \abs{ \sigma^{\frac{1}{m}} }^2 \cdot h       < \infty 
\end{align*}
 
\noindent when $ 1 \le p \le m-1$. Note that we have the cancellation of the length of $ {\sigma}^{k} $ in the fraction of the first equality. For the case of $p=0$, we have the same finiteness having $ C_1 \int_{Z} \abs{ \sigma^{\frac{1}{m}} }^2 \cdot h  $ multiplied by $ \displaystyle \max_Z   \abs{\sigma}^2_{(g_K g_L)^m} $.

\noindent Thus, by Theorem~\ref{main}, there exists $ \st_j^{(km+p)} $ on $X$ satisfying (1) such that

 $$ \int_X   \abs{   \st_j^{(km+p)}  }^2 \cdot g \cdot b   \le  C \int_Z   \abs{  \sigma^{\otimes k} \otimes \st_j^{(p)}|_Z  }^2 \cdot h \cdot b|_Z.$$

\noindent Summing over $j$, we get the following for $ 1 \le p \le m-1 $ (with the obvious modification when $ p=0 $): 

\begin{align*}
 \int_{X \setminus H} \frac{ \sum_{j=1}^{N_p} \abs{ \st_j^{(km+p)}   }^2_{g^{(km+p)}}          }{      \sum_{j=1}^{N_{p-1}}  \abs{ \st_j^{(km+p-1)}   }^2_{g^{(km+p-1)}    }}   dV   &\le   C  \sum_{j=1}^{N_p} \int_Z \abs{  \sigma^{\otimes k} \otimes \st_j^{(p)}|_Z  }^2 \cdot h \cdot b|_Z  \\
&\le C \cdot C_1  \int_{Z} \abs{ \sigma^{\frac{1}{m}} }^2 \cdot h  
\end{align*}

\noindent where $dV$ is a volume form on $X \setminus H$ given by the fact that $g$ is bounded away from zero. Take the constant $ C_\diamondsuit := \max(1,\displaystyle \max_Z   \abs{\sigma}^2_{(g_K g_L)^m}) \cdot C \cdot C_1  \int_{Z} \abs{ \sigma^{\frac{1}{m}} }^2 \cdot h  $ for (*) in Proposition~\ref{extend sigma}. 

\end{proof}

\begin{remark}\label{remark1} 

 In an earlier version of this paper, Theorem~\ref{pluriadjoint} was stated without the hypothesis of $L$ being an integral line bundle, which was incorrect. It had resulted from an incorrect statement of Theorem~\ref{main} (now corrected) without the hypothesis of $L+B$ being an integral line bundle, which we actually needed to define the $\db$ operators in the proof. 

\end{remark}

{

\footnotesize

\bibliographystyle{amsplain}

\begin{thebibliography}{widest-}

\baselineskip0.7mm

\bibitem[AS]{AS} U. Angehrn and Y.-T. Siu, Effective freeness and point separation for adjoint bundles, Invent. math., 122 (1995), 291-308. 

\bibitem[B]{B} B. Berndtsson, The extension theorem of Ohsawa-Takegoshi and the theorem of Donnelly-Fefferman.  Ann. Inst. Fourier (Grenoble)  46  (1996),  no. 4, 1083-1094.


\bibitem[BP]{BP} B. Berndtsson and M. P\u aun,  Bergman kernels and the pseudoeffectivity of relative canonical bundles,  arXiv:math/0703344.

\bibitem[BCHM]{BCHM} C. Birkar, P. Cascini, C. Hacon and J. McKernan, Existence of minimal models for varieties of log general type, arXiv:math/0610203. 



\bibitem[D94]{D94} J.-P. Demailly, $L^2$ vanishing theorems for positive line bundles and adjunction theory , Transcendental methods in algebraic geometry (Cetraro, 1994), 1--97, Lecture Notes in Math., 1646, Springer, Berlin, 1996. 

\bibitem[D97b]{D97b} J.-P. Demailly, Complex analytic and differential geometry, book available at http://www-fourier.ujf-grenoble.fr/~demailly/books.html. 


\bibitem[D00]{D97} J.-P. Demailly, On the Ohsawa-Takegoshi-Manivel $L\sp 2$ extension theorem. Complex analysis and geometry (Paris, 1997), 47--82, Progr. Math., 188, Birkhauser, Basel, 2000. 

\bibitem[D07]{D07} J.-P. Demailly, Ka"hler manifolds and transcendental techniques in algebraic geometry.  International Congress of Mathematicians. Vol. I,  153--186, Eur. Math. Soc., Zu"rich, 2007.


\bibitem[EL]{EL} L. Ein and R. Lazarsfeld, A geometric effective Nullstellensatz.  Invent. Math.  137  (1999),  no. 2, 427--448. 


\bibitem[EP]{EP} L. Ein and M. Popa, Extension of sections and mixed adjoint ideals, preprint. 

\bibitem[F99]{F99} O. Fujino, Applications of Kawamata's positivity theorem, Proc. Japan Acad. Ser. A Math. Sci. 75 (1999), no. 6, 75--79.

\bibitem[F06m]{F06m} O. Fujino, A Memorandum on the invariance of plurigenera (private note), March 2006. 


\bibitem[GR]{GR} R. Gunning and H. Rossi, Analytic functions of several complex variables, Prentice-Hall, Inc. 


\bibitem[GR2]{GR2} H. Grauert and R. Remmert, Coherent analytic sheaves, Grundlehren der mathematischen Wissenschaften 265, Spinger-Verlag, 1984. 


\bibitem[GH]{GH} P. Griffiths and J. Harris, Principles of algebraic geometry, John Wiley and Sons, Inc, 1978.

\bibitem[H]{H} R. Hartshorne, Algebraic geometry, Graduate Texts in Mathematics, No. 52. Springer-Verlag, New York-Heidelberg, 1977.   
  
  

\bibitem[HM06]{HM06} C. Hacon and J. McKernan, Boundedness of pluricanonical maps of varieties of general type. Invent. Math. 166 (2006), no. 1, 1--25.

\bibitem[Ho65]{Ho65} L. H\"ormander, $L^2$ estimates and existence theorems for the $\bar \partial $ operator. Acta Math. 113 1965 89--152. 


\bibitem[Ka97]{Ka97} Y. Kawamata, On Fujita's freeness conjecture for $3$-folds and $4$-folds, Math. Ann. 308 (1997), no. 3, 491--505.

\bibitem[Ka98]{Ka98} Y. Kawamata, Subadjunction of log canonical divisors. II. Amer. J. Math. 120 (1998), no. 5, 893--899. 


\bibitem[Ko97]{Ko97} J. Koll\'ar, Singularities of pairs. Algebraic geometry-Santa Cruz 1995, 221--287, Proc. Sympos. Pure Math., 62, Part 1, Amer. Math. Soc., Providence, RI, 1997.


\bibitem[Ko05]{Ko05} J. Koll\'ar, Kodaira's canonical bundle formula and subadjunction, Chapter 8 of the book Flips for 3-folds and 4-folds, edited by Alessio Corti.




\bibitem[L]{L} R. Lazarsfeld, Positivity in Algebraic Geometry,  Ergebnisse der Mathematik und ihrer Grenzgebiete. 3. Folge. 48 - 49. Springer-Verlag, Berlin, 2004.

\bibitem[Le]{Le} P. Lelong, Plurisubharmonic functions and positive differential forms, Gordon and Breach, New York, and Dunod, Paris, 1969. 


\bibitem[Ma]{Ma} L. Manivel, Un th\'eor\`eme de prolongement $L\sp 2$ de sections holomorphes d'un fibr\'e hermitien,  Math. Z.  212  (1993),  no. 1, 107--122.

\bibitem[M]{M} J. McNeal,  $L\sp 2$ estimates on twisted Cauchy-Riemann complexes.  150 years of mathematics at Washington University in St. Louis,  83--103, Contemp. Math., 395, Amer. Math. Soc., Providence, RI, 2006.

\bibitem[MV]{MV} J. McNeal and D. Varolin, Analytic inversion of adjunction: $L^2$ extension theorems with gain,  Ann. Inst. Fourier (Grenoble)  57  (2007),  no. 3, 703--718. 

\bibitem[Oh95]{Oh95} T. Ohsawa, On the extension of $L\sp 2$ holomorphic functions. III. Negligible weights.  Math. Z.  219  (1995),  no. 2, 215--225. 

\bibitem[OT]{OT} T. Ohsawa and K. Takegoshi, On the extension of $L\sp 2$ holomorphic functions. Math. Z. 195 (1987), no. 2, 197--204. 

\bibitem[Pa07]{Pa05} M. P\u aun, Siu's Invariance of Plurigenera: a One-Tower Proof,  J. Differential Geom.  76  (2007),  no. 3, 485--493.

\bibitem[Ra]{Ra} T. Ransford, Potential theory in the complex plane, Cambridge University Press, Cambridge, 1995.

\bibitem[Ru]{Ru} W. Rudin, Functional Analysis, Second Edition, McGraw-Hill, Inc. 

\bibitem[S96]{Siu96} Y.-T. Siu, The Fujita conjecture and the extension theorem of Ohsawa-Takegoshi, Geometric complex analysis (Hayama, 1995),  577--592, World Sci. Publ., River Edge, NJ, 1996. 

\bibitem[S98]{Siu98} Y.-T. Siu, Invariance of plurigenera, Invent. math., 134 (1998), 661-673. 

\bibitem[S02]{Siu02} Y.-T. Siu, Extension of twisted pluricanonical sections with plurisubharmonic weight and invariance of semipositively twisted plurigenera for manifolds not necessarily of general type. Complex geometry (Gottingen, 2000), 223--277, Springer, Berlin, 2002.

\bibitem[S06]{Siu06} Y.-T. Siu, A general non-vanishing theorem and an analytic proof of the finite generation of the canonical ring, arXiv:math.AG/0610740.   

\bibitem[T06]{T06} S. Takayama, Pluricanonical systems on algebraic varieties of general type, Invent. math.,165 (2006), 551-587. 


\bibitem[V]{Var} D. Varolin, A Takayama-type Extension Theorem, to appear in Compositio Mathematica. 


\bibitem[V07]{Var07} D. Varolin, Analytic Methods in Algebraic Geometry, lecture notes, Stony Brook University, Spring 2007. 

\end{thebibliography}

}  

\qa
\\

\small

\indent\textsc{ Department of Mathematics, University of Chicago, 5734 S. University Avenue, Chicago, IL 60637, USA}

\indent \textit{E-mail address}: \texttt{danokim@math.uchicago.edu}

\end{document}